\documentclass[10pt,leqno]{article}

\usepackage{amsmath,amssymb,amsthm,mathrsfs,dsfont}

\usepackage[margin=3cm]{geometry} 

\usepackage{titlesec,hyperref}

\usepackage{color}

\usepackage{fancyhdr}
\titleformat{\subsection}{\it}{\thesubsection.\enspace}{1pt}{}

\newtheorem{theo}{Theorem}[section]
\newtheorem{lemm}[theo]{Lemma}

\newtheorem{coro}[theo]{Corollary}
\newtheorem{prop}[theo]{Proposition}
\newtheorem{rema}[theo]{Remark}
\numberwithin{equation}{section}

\allowdisplaybreaks 

\newcommand\lm{{\lesssim}}

\begin{document}
\title{Global strong solutions and optimal $L^2$ decay to the compressible FENE dumbbell model
\hspace{-4mm}
}

\author{ Zhaonan $\mbox{Luo}^1$ \footnote{email: 1411919168@qq.com},\quad
Wei $\mbox{Luo}^1$\footnote{E-mail:  luowei23@mail2.sysu.edu.cn} \quad and\quad
 Zhaoyang $\mbox{Yin}^{1,2}$\footnote{E-mail: mcsyzy@mail.sysu.edu.cn}\\
 $^1\mbox{Department}$ of Mathematics,
Sun Yat-sen University, Guangzhou 510275, China\\
$^2\mbox{Faculty}$ of Information Technology,\\ Macau University of Science and Technology, Macau, China}

\date{}
\maketitle
\hrule

\begin{abstract}
In this paper, we are concerned with the global well-posedness and $L^2$ decay rate for the strong solutions of the compressible finite extensible nonlinear elastic (FENE) dumbbell model. For $d\geq 2$, we prove that the compressible FENE dumbbell model admits a unique global strong solution provided the initial data are close to equilibrium state. Moreover, by the Littlewood-Paley decomposition theory and the Fourier splitting method, we show optimal $L^2$ decay rate of global strong solutions for $d\geq 3$.  \\
\vspace*{5pt}

\vspace*{5pt}
\end{abstract}

\vspace*{10pt}

\tableofcontents

\section{Introduction}
In this paper we study the compressible finite extensible nonlinear elastic (FENE) dumbbell model \cite{Bird1977,Doi1988}:
\begin{align}\label{eq0}
\left\{
\begin{array}{ll}
\varrho_t+div(\varrho u)=0 , \\[1ex]
(\varrho u)_t+div(\varrho u\otimes u)-div\Sigma{(u)}+\nabla_x{P(\varrho)}=div~\tau, \\[1ex]
\psi_t+u\cdot\nabla\psi=div_{R}[- \sigma(u)\cdot{R}\psi+\nabla_{R}\psi+\nabla_{R}\mathcal{U}\psi],  \\[1ex]
\tau_{ij}=\int_{B}(R_{i}\nabla_{j}\mathcal{U})\psi dR, \\[1ex]
\varrho|_{t=0}=\varrho_0,~~u|_{t=0}=u_0,~~\psi|_{t=0}=\psi_0, \\[1ex]
(\nabla_{R}\psi+\nabla_{R}\mathcal{U}\psi)\cdot{n}=0 ~~~~ \text{on} ~~~~ \partial B(0,R_{0}) .\\[1ex]
\end{array}
\right.
\end{align}
In dumbbell model \eqref{eq0}, $\varrho(t,x)$ denotes the density of the solvent, $u(t,x)$ represents the velocity of the polymeric liquid and $\psi(t,x,R)$ is the distribution function for the internal configuration. A polymer is described as an "elastic dumbbell" consisting of two "beads" joined by a spring which can be modeled by the polymer elongation $R$. The finite extensibility of the polymers means that $R$ satisfies $R\in B=B(0,R_{0})$. Let $x\in\mathbb{R}^d$ and $\psi(t,x,R)$ satisfy $\int_{B} \psi(t,x,R)dR =1$. The stress tensor $\Sigma{(u)}=\mu(\nabla u+\nabla^{T} u)+\mu'div~u\cdot Id$ satisfies $\mu>0$ and $2\mu+\mu'>0$.
The pressure obeys the so-called $\gamma$-law: $P(\varrho)=\varrho^\gamma$ with $\gamma\geq1$. $\tau$ is an additional stress tensor. Moreover the potential $\mathcal{U}(R)=-k\log(1-(\frac{|R|}{|R_{0}|})^{2})$ for some constant $k>0$. $\sigma(u)=\nabla u$ is the drag term.
This is a micro-macro model (For more details, one can refer to \cite{2021Global}, \cite{Masmoudi2008} and \cite{Masmoudi2013}).

%
%
%


Without loss of generality, we will take $R_{0}=1$ in this paper.
 One can easily to check that the system \eqref{eq0} possesses a trivial solution  $\varrho=1$, $u=0$ and $$\psi_{\infty}(R)=\frac{e^{-\mathcal{U}(R)}}{\int_{B}e^{-\mathcal{U}(R)}dR}=\frac{(1-|R|^2)^k}{\int_{B}(1-|R|^2)^kdR}.$$
By taking the perturbations near the global equilibrium:
\begin{align*}
\rho=\varrho-1,~~u=u,~~g=\frac {\psi-\psi_\infty} {\psi_\infty},
\end{align*}
 we rewrite \eqref{eq0} as the following system:
\begin{align}\label{eq1}
\left\{
\begin{array}{ll}
\rho_t+div~u(1+\rho)=-u\cdot\nabla\rho , \\[1ex]
u_t-\frac 1 {1+\rho} div\Sigma{(u)}+\frac {P'(1+\rho)} {1+\rho} \nabla\rho=-u\cdot\nabla u+\frac 1 {1+\rho} div~\tau, \\[1ex]
g_t+\mathcal{L}g=-u\cdot\nabla g-\frac 1 {\psi_\infty}\nabla_R\cdot(\nabla uRg\psi_\infty)-div~u-\nabla u R\nabla_{R}\mathcal{U},  \\[1ex]
\tau_{ij}(g)=\int_{B}(R_{i}\nabla_{Rj}\mathcal{U})g\psi_\infty dR, \\[1ex]
\rho|_{t=0}=\rho_0,~~u|_{t=0}=u_0,~~g|_{t=0}=g_0, \\[1ex]
\psi_\infty\nabla_{R}g\cdot{n}=0 ~~~~ \text{on} ~~~~ \partial B(0,1) ,\\[1ex]
\end{array}
\right.
\end{align}
where $\mathcal{L}g=-\frac 1 {\psi_\infty}\nabla_R\cdot(\psi_\infty\nabla_{R}g)$.

{\bf Remark.} As in the reference \cite{Masmoudi2013}, one can deduce that $\psi=0$ on $\partial B(0,1)$.

M. Renardy \cite{Renardy} established the local well-posedness in Sobolev spaces with potential $\mathcal{U}(R)=(1-|R|^2)^{1-\sigma}$ for $\sigma>1$. Later, B. Jourdain, T. Leli\`{e}vre, and
C. Le Bris \cite{Jourdain} proved local existence of a stochastic differential equation with potential $\mathcal{U}(R)=-k\log(1-|R|^{2})$ in the case $k>3$ for a Couette flow. H. Zhang and P. Zhang \cite{Zhang-H} proved local well-posedness of (1.4) with $d=3$ in weighted Sobolev spaces. For the co-rotation case, F. Lin, P. Zhang, and Z. Zhang \cite{F.Lin} obtained a global existence results with $d=2$ and $k > 6$. If the initial data is perturbation around equilibrium, N. Masmoudi \cite{Masmoudi2008} proved global well-posedness of (1.4) for $k>0$. In the co-rotation case with $d=2$, he \cite{Masmoudi2008} obtained a global result for $k>0$ without any small conditions. In the co-rotation case, A. V. Busuioc, I. S. Ciuperca, D. Iftimie and L. I. Palade \cite{Busuioc} obtained a global existence result with only the small condition on $\psi_0$. The global existence of weak solutions in $L^2$ was proved recently by N. Masmoudi \cite{Masmoudi2013} under some entropy conditions.
Recently, M. Schonbek \cite{Schonbek} studied the $L^2$ decay of the velocity for the co-rotation
FENE dumbbell model, and obtained the
decay rate $(1+t)^{-\frac{d}{4}+\frac{1}{2}}$, $d\geq 2$ with $u_0\in L^1$.
Moreover, she conjectured that the sharp decay rate should be $(1+t)^{-\frac{d}{4}}$,~$d\geq 2$.
However, she failed to get it because she could not use the bootstrap argument as in \cite{Schonbek1985} due to the
additional stress tensor.  Recently, W. Luo and Z. Yin \cite{Luo-Yin} improved Schonbek's result
and showed that the decay rate is $(1+t)^{-\frac{d}{4}}$ with $d\geq 3$ and $\ln^{-l}(1+t)$ with $d=2$ for any $l\in\mathbb{N^+}$.
This result shows that M. Schonbek's conjecture is true when $d \geq3$. More recently, W. Luo and Z. Yin \cite{Luo-Yin2} improved the decay rate to $(1+t)^{-\frac{d}{4}}$ with $d=2$.

\subsection{Short reviews for the compressible Navier-Stokes (CNS) equations}
The system \eqref{eq0} reduce to the CNS equations by taking $\psi\equiv 0$. In order to study about the \eqref{eq0}, we cite some reference about the CNS equations. The first local existence and uniqueness results were obtained by J. Nash \cite{miaocompressns23} for smooth initial data without vacuum. Later on, A. Matsumura and T. Nishida \cite{Matsumura} proved the global well-posedness and the time decay rate for smooth data close to equilibrium for $d=3$. In \cite{miaocompressns18}, A. V. Kazhikhov and V. V. Shelukhin established the first global existence result with large data in one dimensional space under some suitable condition on $\mu$ and $\lambda$. If $\mu$ is constant and $\lambda(\rho)=b\rho^\beta$, X. Huang and J. Li\cite{Huang} obtained a global existence and uniqueness result for large initial data in two dimensional space(See also \cite{miaocompressns26}). In \cite{miaocompressns25}, X. Huang, J. Li, and Z. Xin proved the global well-posedness with vacuum. The blow-up phenomenons were studied by Z. Xin et al
in \cite{Xin98,miaocompressns28,miaocompressns27}. Concerning the global existence of weak solutions for the large initial data, we may refer to \cite{miaocompressns2,miaocompressns3,miaocompressns21,Vasseur}.

To catch the scaling invariance property of the CNS equations. R. Danchin introduced the "critical spaces" in his series papers \cite{miaocompressns11,miaocompressns12,miaocompressns14,miaocompressns15,miaocompressns155} and obtained several important existence and uniqueness results. Recently, Q. Chen, C. Miao and Z. Zhang \cite{miao} proved the local existence and uniqueness in critical homogeneous Besov spaces. The ill-posedness result was obtained in \cite{miaoill-posedness}. In \cite{miaocompressns24}, L. He, J. Huang and C. Wang proved the global stability with $d=3$ i.e. for any perturbed solutions will remain close to the reference solutions if initially they are close to another one.

The large time behaviour was proved by H. Li and T. Zhang in \cite{Li2011Large}. They obtain the optimal time decay rate for the CNS equations by spectrum analysis in Sobolev spaces. Recently, J. Xu\cite{Xu2019} studied about the large time behaviour in the critical Besov space and obtain the optimal time decay rate.

\subsection{Main results}

J. Ning, Y. Liu and T. Zhang \cite{2017Global} proved the first global well-posedness for \eqref{eq0} if the initial data is close to the equilibrium. In \cite{2017Global}, the authors assume that $R\in\mathbb{R}^3$ which means that polymer elongation can be infinite. Actually, the polymer elongation $R$ is usually bounded.

Recently, N. Masmoudi \cite{2016Equations} is concerning with the long time behavior for polymeric models. The co-rotation compressible FENE system has been studied in \cite{2021Global}. In the co-rotation case, the drag term $\sigma(u)=\frac{\nabla u-\nabla u^T}{2}$ which leads to some good structure such that the $\psi-\psi_\infty$ is exponential decay in time. To our best knowledge, the same problem for the general case has not been studied yet. This problem is interesting and more difficult than the co-rotation case. In this paper, we firstly study the global well-posedness results for \eqref{eq0}. The key point is to prove a global priori estimate for \eqref{eq1} with small data. Using the energy methods and the cancellation relation between the CNS equations and Fokker-Planck equation, for $d\geq2$, we obtain a global priori estimate.
Moreover, if $d\geq3$, we study about the large time behaviour and obtain the optimal time decay rate for $(\rho,u)$ in $L^2$. The proof is based on the Fourier splitting method and the Littlewood-Paley decomposition theory. The first difficult is to estimate the additional linear term $div~\tau$. Motivated by \cite{He2009} and \cite{2018Global}, we can cancel the stress term $\tau$ in Fourier space. Then we obtain the time decay rate $(1+t)^{-\frac{d}{8}}$ for the velocity in $L^2$ by the Fourier splitting method and the bootstrap argument. The main difficult to get optimal time decay rate is that we can not get any information of $u$ in $L^1$ from \eqref{eq1}. Fortunately, similar to \cite{Tong2017The}, we can prove a slightly weaker conclusion $\|u\|_{L^\infty(0,\infty; \dot{B}^{-\frac d 2}_{2,\infty})}\leq C$ from \eqref{eq1} by using the time decay rate $(1+t)^{-\frac{d}{8}}$. Finally, we obtain optimal time decay rate for the velocity in $L^2$ by the Littlewood-Paley decomposition theory and the standard Fourier splitting method.

Our main result can be stated as follows.

Using the energy methods in \cite{2021Global}, one can deduce that the global existence of strong solutions for \eqref{eq1}. However, to obtain optimal time decay rate, we need a more precise higher order derivatives estimate for \eqref{eq1}. We establish the precise higher order derivatives estimate in the proof of the following Theorem.
\begin{theo}[Global well-posedness]\label{th1}
Let $d\geq 2~and~s>1+\frac d 2$. Assume that $(\rho_0,u_0,g_0)\in H^s\times H^s\times H^s(\mathcal{L}^2)$, then there exists a sufficiently small constant $\epsilon_0$ such that if
$\int_B g_{0}\psi_{\infty}dR=0$ and $1+g_0>0$ and
\begin{align}
E(0)=\|\rho_0\|^2_{H^s}+\|u_0\|^2_{H^s}+\|g_0\|^2_{H^s(\mathcal{L}^2)}\leq \epsilon_0,
\end{align}
then the compressible FENE system \eqref{eq1} admits a unique global strong solution $(\rho,u,g)$ satisfying $\int_B g\psi_{\infty}dR=0$ and $1+g>0$ and
\begin{align}
\sup_{t\in[0,+\infty)} E(t)+\int_{0}^{\infty}D(t)dt\leq C_0 E(0),
\end{align}
where $C_0>1$ is a constant.
\end{theo}

\begin{theo}[Large time behaviour]\label{th2}
Let $d\geq 3$. Assume that $(\rho_0,u_0,g_0)$ satisfy the condition in Theorem \ref{th1}, in addition, if $(\rho_0,u_0)\in \dot{B}^{-\frac d 2}_{2,\infty}\times \dot{B}^{-\frac d 2}_{2,\infty}$ and $g_0\in \dot{B}^{-\frac d 2}_{2,\infty}(\mathcal{L}^2)$, then the corresponding solution $(\rho,u,g)$ satisfy
\begin{align}\label{decay}
\|\rho\|_{L^2}+\|u\|_{L^2}\leq C(1+t)^{-\frac d 4}
\end{align}
and
\begin{align}
\|g\|_{L^2(\mathcal{L}^2)}\leq C(1+t)^{-\frac{d}{4}-\frac{1}{2}}.
\end{align}
\end{theo}

\begin{rema}
Taking $\psi \equiv 0$ and combining with the result in \cite{Li2011Large}, we can see that the $L^2$ decay rate for $(\rho,u)$ obtained in Theorem \ref{th2} is optimal.
\end{rema}

\begin{rema}
In previous papers, researchers usually add the condition $(\rho_0,u_0)\in L^1\times  L^1$ to obtain the optimal time decay rate. Since $L^1\hookrightarrow \dot{B}^{-\frac d 2}_{2,\infty}$, it follows that our condition is weaker and the results still hold true for $(\rho_0, u_0)\in L^1\times  L^1$.  Moreover, the assumption can be replaced with a weaker assumption $\sup_{j\leq j_0}2^{-\frac d 2 j}\|\dot{\Delta}_j (\rho_0,u_0,g_0)\|_{L^2\times L^2\times L^2(\mathcal{L}^2)} <\infty$, for any $j_0\in \mathbb{Z}$.
\end{rema}

The paper is organized as follows. In Section 2 we introduce some notations and preliminaries which will be used in the sequel. In Section 3 we prove the global well-posedness of the compressible FENE dumbbell model for $d\geq2$. In Section 4 we study the optimal $L^2$ decay of solutions to the compressible FENE model by using the Fourier splitting method, the bootstrap argument and the Littlewood-Paley decomposition theory $d\geq3$.

\section{Preliminaries}
For the convenience of readers, we give some notations and useful lemmas in this section .

Let $p\geq1$. We denote by $\mathcal{L}^{p}$ the space
$$\mathcal{L}^{p}=\big\{f \big|\|f\|^{p}_{\mathcal{L}^{p}}=\int_{B} \psi_{\infty}|f|^{p}dR<\infty\big\},$$
and denote by $L^{p}_{x}(\mathcal{L}^{q})$ the space
$$L^{p}_{x}(\mathcal{L}^{q})=\big\{f \big|\|f\|_{L^{p}_{x}(\mathcal{L}^{q})}=(\int_{\mathbb{R}^{d}}(\int_{B} \psi_{\infty}|f|^{q}dR)^{\frac{p}{q}}dx)^{\frac{1}{p}}<\infty\big\}.$$

The symbol $\widehat{f}=\mathcal{F}(f)$ stands for the Fourier transform of $f$.
Let $\Lambda^s f=\mathcal{F}^{-1}(|\xi|^s \widehat{f})$.
If $s\geq0$, we denote by $H^{s}(\mathcal{L}^{2})$ the space
$$H^{s}(\mathcal{L}^{2})=\{f\big| \|f\|^2_{H^{s}(\mathcal{L}^{2})}=\int_{\mathbb{R}^{d}}\int_B(|f|^2+|\Lambda^s f|^2)\psi_\infty dRdx<\infty\}.$$

Denote that
$$E(t)=\|\rho\|^2_{H^{s}}+\|u\|^2_{H^{s}}+\|g\|^2_{H^{s}(\mathcal{L}^{2})},$$
and
$$D(t)=\|\nabla\rho\|^2_{H^{s-1}}+\mu\|\nabla u\|^2_{H^{s}}+(\mu+\mu')\|div~u\|^2_{H^{s}}+\|\nabla_R g\|^2_{H^{s}(\mathcal{L}^{2})}.$$

We now recall the Littlewood-Paley decomposition theory in the following Proposition.
\begin{prop}\cite{Bahouri2011}\label{pro0}
Let $\mathcal{C}$ be the annulus $\{\xi\in\mathbb{R}^d:\frac 3 4\leq|\xi|\leq\frac 8 3\}$. There exist radial function $\varphi$, valued in the interval $[0,1]$, belonging respectively to $\mathcal{D}(\mathcal{C})$, and such that
$$ \forall\xi\in\mathbb{R}^d\backslash\{0\},\ \sum_{j\in\mathbb{Z}}\varphi(2^{-j}\xi)=1, $$
$$ |j-j'|\geq 2\Rightarrow\mathrm{Supp}\ \varphi(2^{-j}\cdot)\cap \mathrm{Supp}\ \varphi(2^{-j'}\cdot)=\emptyset. $$
Further, we have
$$ \forall\xi\in\mathbb{R}^d\backslash\{0\},\ \frac 1 2\leq\sum_{j\in\mathbb{Z}}\varphi^2(2^{-j}\xi)\leq 1. $$
\end{prop}

Let $u$ be a tempered distribution in $\mathcal{S}'_h(\mathbb{R}^d)$. For all $j\in\mathbb{Z}$, define
$$
\dot{\Delta}_j u=\mathcal{F}^{-1}(\varphi(2^{-j}\cdot)\mathcal{F}u).
$$
Then the Littlewood-Paley decomposition is given as follows:
$$ u=\sum_{j\in\mathbb{Z}}\dot{\Delta}_j u \quad \text{in}\ \mathcal{S}'(\mathbb{R}^d). $$

Let $s\in\mathbb{R},\ 1\leq p,r\leq\infty.$ The homogeneous Besov space $\dot{B}^s_{p,r}$ and $\dot{B}^s_{p,r}(\mathcal{L}^q)$ are defined by
$$ \dot{B}^s_{p,r}=\{u\in \mathcal{S}'_h:\|u\|_{\dot{B}^s_{p,r}}=\Big\|(2^{js}\|\dot{\Delta}_j u\|_{L^p})_j \Big\|_{l^r(\mathbb{Z})}<\infty\}, $$
$$ \dot{B}^s_{p,r}(\mathcal{L}^q)=\{\phi\in \mathcal{S}'_h:\|\phi\|_{\dot{B}^s_{p,r}(\mathcal{L}^q)}=\Big\|(2^{js}\|\dot{\Delta}_j \phi\|_{L_{x}^{p}(\mathcal{L}^q)})_j \Big\|_{l^r(\mathbb{Z})}<\infty\}.$$

We agree that $f\lm g$ represents $f\leq Cg$ with a constant $C$ and $\nabla$ stands for $\nabla_x$ and $div$ stands for $div_x$.

The following lemma is the Gagliardo-Nirenberg inequality of Sobolev type.
\begin{lemm}\cite{1959On}\label{Lemma0}
Let $d\geq2,~p\in[2,+\infty)$ and $0\leq s,s_1\leq s_2$, then there exists a constant $C$ such that
 $$\|\Lambda^{s}f\|_{L^{p}}\leq C \|\Lambda^{s_1}f\|^{1-\theta}_{L^{2}}\|\Lambda^{s_2} f\|^{\theta}_{L^{2}},$$
where $0\leq\theta\leq1$ and $\theta$ satisfy
$$ s+d(\frac 1 2 -\frac 1 p)=s_1 (1-\theta)+\theta s_2.$$
Note that we require that $0<\theta<1$, $0\leq s_1\leq s$, when $p=\infty$.
\end{lemm}

The following lemmas are useful for estimating $\tau$.
\begin{lemm}\cite{Masmoudi2008}\label{Lemma1}
 If $\int_B g\psi_\infty dR=0$, then there exists a constant $C$ such that
 $$\|g\|_{\mathcal{L}^{2}}\leq C \|\nabla _{R} g\|_{\mathcal{L}^{2}}.$$
\end{lemm}

\begin{lemm}\label{Lemma2}
\cite{Masmoudi2008} For any $\delta>0$, there exists a constant $C_{\delta}$ such that
$$|\tau(g)|^2\leq\delta\|\nabla _{R}g\|^2_{\mathcal{L}^{2}}
+C_{\delta}\|g\|^2_{\mathcal{L}^{2}}.$$  \\
If $(p-1)k>1$, then
$$|\tau(g)|\leq C\|g\|_{\mathcal{L}^{p}}.$$
\end{lemm}

To get the optimal $L^2$ decay rate, we need a more precise estimate of $\tau$.
\begin{lemm}\cite{2021Global}\label{Lemma3}
If $\|g\|_{\mathcal{L}^{2}}\leq C \|\nabla _{R} g\|_{\mathcal{L}^{2}}<\infty$, there exists a constant $C_{1}$ such that
\begin{align}\label{ha1}
|\tau(g)|\leq C_{1}\|g\|^{\frac {k+1} 2}_{\mathcal{L}^{2}}
\|\nabla _{R}g\|^{\frac {1-k} 2}_{\mathcal{L}^{2}},~~for~0<k<1,
\end{align}
and
\begin{align}\label{ha2}
|\tau(g)|\leq C_{1}\|g\|^{\frac {2n} {2n+1}}_{\mathcal{L}^{2}}
\|\nabla _{R}g\|^{\frac {1} {2n+1}}_{\mathcal{L}^{2}},~~for~k=1~and~\forall n\geq1.
\end{align}
\end{lemm}

\begin{coro}\label{rema}
According to Lemma \ref{Lemma1}-\ref{Lemma3}, if $\int_B g\psi_\infty dR=0$, for any $k>0$, we have $|\tau(g)|\leq C_{1}\|g\|^{\frac 1 2}_{\mathcal{L}^{2}}
\|\nabla _{R}g\|^{\frac 1 2}_{\mathcal{L}^{2}}$.
\end{coro}

\begin{lemm}\cite{Moser1966A}\label{Lemma4}
Let $s\geq 1$, $p,p_1,p_4\in (1,\infty)$ and $\frac 1 p =\frac 1 {p_1}+\frac 1 {p_2}=\frac 1 {p_3}+\frac 1 {p_4}$, then there exists a constant $C$ such that
$$\|[\Lambda^s, f]g\|_{L^p}\leq C(\|\Lambda^{s}f\|_{L^{p_1}}\|g\|_{L^{p_2}}+\|\nabla f\|_{L^{p_3}}\|\Lambda^{s-1}g\|_{L^{p_4}}),$$
and
$$\|[\Lambda^s, f]g\|_{L^2(\mathcal{L}^{2})}\leq C(\|\Lambda^{s}f\|_{L^2}\|g\|_{L^\infty(\mathcal{L}^{2})}+\|\nabla f\|_{L^\infty}\|\Lambda^{s-1}g\|_{L^2(\mathcal{L}^{2})}).$$
\end{lemm}

\section{Global strong solutions with small data}
In this section, we investigate the global well-posedness for the compressible FENE dumbbell model with $d\geq2$.
We divide the proof of Theorem \ref{th1} into two Propositions. Using the standard iterating method in \cite{2017Global} and \cite{2021Global}, one can easily deduce that the existence of local solutions. Thus we omit the proof here and present the following Proposition.
\begin{prop}\label{pro1}
Let $d\geq 2~and~s>1+\frac d 2$. If $E(0)\leq \frac {\epsilon} 2$, then there exist a time $T>0$ such that \eqref{eq1} admits a unique local strong solution $(\rho_,u,g)\in L^{\infty}(0,T;H^s\times H^s\times H^s(\mathcal{L}^2))$ and we get
\begin{align}
\sup_{t\in[0,T]} E(t)+\int_{0}^{T}H(t)dt\leq \epsilon,
\end{align}
where $H(t)=\mu\|\nabla u\|^2_{H^{s}}+(\mu+\mu')\|div~u\|^2_{H^{s}}+\|\nabla_R g\|^2_{H^{s}(\mathcal{L}^{2})}.$
\end{prop}
Denote that
$$E_\eta(t)=\sum_{n=0,s}(\|h(\rho)^{\frac 1 2}\Lambda^n\rho\|^2_{L^{2}}+\|(1+\rho)^{\frac 1 2}\Lambda^nu\|^2_{L^{2}})+\lambda\|g\|^2_{H^{s}(\mathcal{L}^{2})}+2\eta\sum_{m=0,s-1}\int_{\mathbb{R}^{d}} \Lambda^m u\nabla\Lambda^m \rho dx,$$
and
$$D_\eta(t)=\eta\gamma\|\nabla\rho\|^2_{H^{s-1}}+\mu\|\nabla u\|^2_{H^{s}}+(\mu+\mu')\|div~u\|^2_{H^{s}}+\lambda\|\nabla_R g\|^2_{H^{s}(\mathcal{L}^{2})}.$$
In the following proposition, we prove a key global priori estimate for \eqref{eq1}.
\begin{prop}\label{pro2}
Let $d\geq 2~and~s>1+\frac d 2$. Assume that $(\rho_,u,g)\in L^{\infty}(0,T;H^s\times H^s\times H^s(\mathcal{L}^2))$ are local strong solutions constructed in Proposition \ref{pro1}. If $\sup_{t\in[0,T)} E(t)\leq \epsilon$, then we have
\begin{align}
\frac d {dt} E_\eta(t)+D_\eta(t)\leq 0,
\end{align}
and there exist a constant $C_0>1$ such that
\begin{align}
\sup_{t\in[0,T]} E(t)+\int_{0}^{T}D(t)dt\leq C_{0}E(0).
\end{align}
\end{prop}
\begin{proof}
By virtue of the coupling effect between $(\rho, u, g)$ and using the energy methods, we can easily deduce that the lower order derivatives estimates for $\eqref{eq1}$, see \cite{2021Global}.

Multiplying $\psi_\infty$ to $\eqref{eq1}_3$ and integrating over $B$ with $R$, we obtain $\int_{B}g\psi_\infty dR=\int_{B}g_0\psi_\infty dR=0$.
$L^2(\mathcal{L}^{2})$ inner product is denoted by $\langle f,g\rangle=\int_{\mathbb{R}^{d}}\int_{B}fg\psi_\infty dRdx$. Since $u$ is independent on $R$, it follows that $\langle div~u,g\rangle=0$. Taking the $L^2(\mathcal{L}^{2})$ inner product with $g$ to $\eqref{eq1}_3$, we obtain
\begin{align}
\frac {1} {2}\frac {d} {dt} \|g\|^2_{L^2(\mathcal{L}^{2})}+\|\nabla_R g\|^2_{L^2(\mathcal{L}^{2})}-\int_{\mathbb{R}^{d}}\nabla u:\tau dx  =-\langle u\cdot\nabla g, g\rangle-\langle\frac 1 {\psi_\infty} \nabla_R\cdot(\nabla uRg\psi_\infty), g \rangle.
\end{align}
Integrating by parts, we have
\begin{align*}
-\langle u\cdot\nabla g,g\rangle=\frac 1 2 \langle divu, g^2\rangle
\lesssim \|\nabla u\|_{L^\infty}\|g\|^2_{L^2(\mathcal{L}^{2})},
\end{align*}
and
\begin{align*}
\langle\frac 1 {\psi_\infty} \nabla_R\cdot(\nabla uRg\psi_\infty), g \rangle=
-\int_{\mathbb{R}^{d}}\int_{B}(\nabla uRg\psi_\infty)\nabla_R g dRdx
\lesssim \|\nabla u\|_{L^\infty}\|g\|_{L^2(\mathcal{L}^{2})}\|\nabla_R g\|_{L^2(\mathcal{L}^{2})}.
\end{align*}
Applying Lemma \ref{Lemma1}, we have $\|g\|^2_{L^2(\mathcal{L}^{2})}\lesssim\|\nabla_R g\|^2_{L^2(\mathcal{L}^{2})}$, which implies that
\begin{align}\label{co1}
\frac {1} {2}\frac {d} {dt} \|g\|^2_{L^2(\mathcal{L}^{2})}+\|\nabla_R g\|^2_{L^2(\mathcal{L}^{2})}+\int_{\mathbb{R}^{d}}\nabla u:\tau dx\lesssim \|\nabla u\|_{L^\infty}\|g\|_{L^2(\mathcal{L}^{2})}\|\nabla_R g\|_{L^2(\mathcal{L}^{2})}.
\end{align}
Denote that $h(\rho)=\frac {P'(1+\rho)} {1+\rho}$ and $i(\rho)=\frac 1 {\rho+1}$. Multiplying $h(\rho) \rho$ to $\eqref{eq1}_1$ and integrating over $\mathbb{R}^{d}$ with $x$, we obtain
\begin{align}\label{co2}
&\frac 1 2 \frac {d} {dt} \int_{\mathbb{R}^{d}}h(\rho)|\rho|^2 dx+\int_{\mathbb{R}^{d}}P'(1+\rho)\rho div~u dx \\ \notag
&=\frac 1 2 \int_{\mathbb{R}^{d}} \partial_t h(\rho) |\rho|^2 dx-\int_{\mathbb{R}^{d}}h(\rho)\rho u\cdot\nabla \rho dx.
\end{align}
Multiplying $(1+\rho)u$ to $(\ref{eq1})_2$ and integrating over $\mathbb{R}^{d}$ with $x$, we deduce that
\begin{align}\label{co3}
&\frac 1 2 \frac {d} {dt} \int_{\mathbb{R}^{d}}(1+\rho)|u|^2 dx+\int_{\mathbb{R}^{d}}P'(1+\rho)u\nabla\rho dx-\int_{\mathbb{R}^{d}}udiv\Sigma(u)dx-\int_{\mathbb{R}^{d}}udiv~\tau dx  \\ \notag
&=\frac 1 2 \int_{\mathbb{R}^{d}} \partial_t\rho|u|^2 dx-\int_{\mathbb{R}^{d}}u\cdot\nabla u (1+\rho)u dx.
\end{align}
Using Lemma \ref{Lemma0}, we get
\begin{align*}
\begin{split}
&\frac 1 2 \int_{\mathbb{R}^{d}} \partial_t h(\rho) |\rho|^2 dx
\lesssim \|\nabla\rho\|_{L^2}(\|\nabla \rho\|_{L^2}\|u\|_{L^d}+\|\nabla u\|_{L^2}\|\rho\|_{L^d}), \\
&\int_{\mathbb{R}^{d}}h(\rho)\rho u\cdot\nabla \rho dx
\lesssim \|\nabla\rho\|_{L^2}(\|\nabla u\|_{L^2}\|\rho\|_{L^d}+\|\nabla \rho\|_{L^2}\|u\|_{L^d}),  \\
&\frac 1 2 \int_{\mathbb{R}^{d}} \partial_t\rho|u|^2 dx+\int_{\mathbb{R}^{d}}u\cdot\nabla u (1+\rho)u dx
\lesssim \|\nabla\rho\|_{L^2}\|\nabla u\|_{L^2}\|u\|_{L^d}+\|\nabla u\|^2_{L^2}\|u\|_{L^d}.
\end{split}
\end{align*}
Integrating by parts, we obtain
\begin{align*}
-\int_{\mathbb{R}^{d}}P'(1+\rho)(u\nabla\rho+\rho div~u)dx
=\int_{\mathbb{R}^{d}}P''(1+\rho)\rho u\nabla\rho dx\lesssim \|\nabla \rho\|_{L^2}\|\nabla u\|_{L^2}\|\rho\|_{L^d}+\|\nabla \rho\|^2_{L^2}\|u\|_{L^d}.
\end{align*}
Multiplying $\nabla\rho$ to $(\ref{eq1})_2$ and integrating over $\mathbb{R}^{d}$ with $x$, then we get
\begin{align}\label{co4}
&\frac {d} {dt} \int_{\mathbb{R}^{d}} u\nabla\rho dx+\gamma\|\nabla\rho\|^2_{L^2}  \\ \notag
&= \int_{\mathbb{R}^{d}}u\nabla\rho_t dx
+\int_{\mathbb{R}^{d}}\nabla\rho\cdot\{i(\rho) div\Sigma{(u)}-(h(\rho)-\gamma) \nabla\rho-u\cdot\nabla u+i(\rho) div~\tau\} dx  \\ \notag
&=I_1+I_2.
\end{align}
By virtue of integration by parts, we have
\begin{align*}
I_1=-\int_{\mathbb{R}^{d}}div~u\rho_t dx
\lesssim \|\nabla u\|^2_{L^2}(1+\|\rho\|_{L^\infty})+\|\nabla u\|_{L^2}\|\nabla \rho\|_{L^2}\|u\|_{L^\infty}.
\end{align*}
Applying Lemma \ref{Lemma1} and Lemma \ref{Lemma2}, we obtain
\begin{align*}
I_2\lesssim \|\nabla \rho\|_{L^2}(\|\nabla^2 u\|_{L^2}+\|\rho\|_{L^\infty}\|\nabla \rho\|_{L^2}+\|u\|_{L^\infty}\|\nabla u\|_{L^2}+\|\nabla\nabla_R g\|_{L^2(\mathcal{L}^{2})}).
\end{align*}

Let $\eta<1$, which will be chosen later. Combining \eqref{co1} and the estimates for \eqref{co2}-\eqref{co4}, we obtain the lower order derivatives estimates for $\eqref{eq1}$:
\begin{align}\label{low estimate}
&\frac {d} {dt} (\|h(\rho)^{\frac 1 2}\rho \|^2_{L^2}+\|(1+\rho)^{\frac 1 2}u\|^2_{L^2}+\|g\|^2_{L^2(\mathcal{L}^{2})}+2\eta\int_{\mathbb{R}^{d}} u\nabla\rho dx)  \\ \notag
&+2(\mu\|\nabla u\|^2_{L^2}+(\mu+\mu')\|divu\|^2_{L^2}+
\eta\gamma\|\nabla\rho\|^2_{L^2}+\|\nabla_R g\|^2_{L^2(\mathcal{L}^{2})})  \\ \notag
&\lesssim (\|\nabla \rho\|^2_{L^2}+\|\nabla u\|^2_{L^2})(\|u\|_{L^d}+\|\rho\|_{L^d})+\|\nabla u\|_{L^\infty}\|\nabla_R g\|^2_{L^2(\mathcal{L}^{2})}+\eta\|\nabla u\|^2_{L^2}(1+\|\rho\|_{L^\infty})  \\ \notag
&+\eta\|\nabla \rho\|_{L^2}(\|\nabla^2 u\|_{L^2}+\|\rho\|_{L^\infty}\|\nabla \rho\|_{L^2}+\|u\|_{L^\infty}\|\nabla u\|_{L^2}+\|\nabla\nabla_R g\|_{L^2(\mathcal{L}^{2})}).
\end{align}

From now on, we establish a precise estimate of the higher order derivatives for \eqref{eq1} by interpolation theory. The estimate will play a important role in improving the time decay rate.\\
Applying $\Lambda^s$ to $(\ref{eq1})_3$, we deduce that
\begin{align}\label{h3}
&\partial_t\Lambda^s g+\mathcal{L}\Lambda^s g+div\Lambda^{s}u+\nabla\Lambda^{s}uR\nabla_R \mathcal{U} \\ \notag
&=-u\cdot\nabla\Lambda^s g-[\Lambda^s,u]\nabla g-\frac 1 {\psi_\infty}\nabla_R \cdot(\Lambda^s\nabla uRg\psi_\infty+R\psi_\infty[\Lambda^s,g]\nabla u).
\end{align}
Taking the $L^2(\mathcal{L}^{2})$ inner product with $\Lambda^s g $ to $(\ref{h3})$, we obtain
\begin{align}
&\frac {1} {2}\frac {d} {dt} \|\Lambda^s g\|^2_{L^2(\mathcal{L}^{2})}+\|\nabla_R  \Lambda^s g\|^2_{L^2(\mathcal{L}^{2})}-\int_{\mathbb{R}^{d}}\nabla\Lambda^{s}u:\Lambda^{s}\tau dx
=-\langle u\cdot\nabla \Lambda^s g,\Lambda^s g\rangle  \\ \notag
&-\langle[\Lambda^s,u]\nabla g,\Lambda^s g\rangle
-\langle\frac 1 {\psi_\infty} \nabla_R\cdot(\Lambda^s\nabla uRg\psi_\infty),\Lambda^sg\rangle
-\langle\frac 1 {\psi_\infty} \nabla_R\cdot(R\psi_\infty[\Lambda^s,g]\nabla u),\Lambda^s g\rangle.
\end{align}
Integrating by part and using Lemma \ref{Lemma4}, we have
\begin{align*}
\begin{split}
-\langle u\cdot\nabla \Lambda^s g,\Lambda^s g\rangle&=\frac 1 2 \langle div~u, (\Lambda^s g)^2\rangle
\lesssim \|\nabla u\|_{L^\infty}\|\Lambda^s g\|^2_{L^2(\mathcal{L}^{2})},\\
-\langle[\Lambda^s,u]\nabla g,\Lambda^s g\rangle&\lesssim \|u\|_{H^s}\|\nabla g\|^2_{H^{s-1}(\mathcal{L}^{2})},\\
\langle\frac 1 {\psi_\infty} \nabla_R\cdot(\Lambda^s\nabla uRg\psi_\infty), \Lambda^sg \rangle
&=-\int_{\mathbb{R}^{d}}\int_{B}(\Lambda^s\nabla uR\psi_\infty g)\nabla_R \Lambda^s g dRdx, \\
&\lesssim \|g\|_{L^\infty(\mathcal{L}^{2})}\|\nabla \Lambda^s u\|_{L^2}\|\nabla_R\Lambda^s g\|_{L^2(\mathcal{L}^{2})},\\
-\langle\frac 1 {\psi_\infty} \nabla_R\cdot(R\psi_\infty[\Lambda^s,g]\nabla u),\Lambda^s g\rangle&=\langle R[\Lambda^s,g]\nabla u,\nabla_R\Lambda^s g\rangle  \\
&\lesssim \|\nabla_R \Lambda^s g\|_{L^2(\mathcal{L}^{2})}\|u\|_{H^s}\|\nabla g\|_{H^{s-1}(\mathcal{L}^{2})},
\end{split}
\end{align*}
from which we can deduce that
\begin{multline}\label{co5}
\frac {1} {2}\frac {d} {dt} \|\Lambda^s g\|^2_{L^2(\mathcal{L}^{2})}+\|\nabla_R \Lambda^s g\|^2_{L^2(\mathcal{L}^{2})}-\int_{\mathbb{R}^{d}}\nabla\Lambda^{s}u:\Lambda^{s}\tau dx \\
\lesssim \|g\|_{L^\infty(\mathcal{L}^{2})}\|\nabla \Lambda^s u\|_{L^2}\|\nabla_R\Lambda^s g\|_{L^2(\mathcal{L}^{2})}+\|u\|_{H^s}\|\nabla_R\nabla g\|^2_{H^{s-1}(\mathcal{L}^{2})}.
\end{multline}
Applying $\Lambda^s$ to $(\ref{eq1})_1$ and applying $\Lambda^m$ to $(\ref{eq1})_2$, we get
\begin{align}\label{h1}
\partial_t\Lambda^s \rho+div \Lambda^s u(1+\rho)
=-u\cdot\nabla\Lambda^s \rho-[\Lambda^s,u]\nabla\rho-[\Lambda^s,\rho]div~u,
\end{align}
and
\begin{align}\label{h2}
&\partial_t\Lambda^m u+h(\rho)\nabla\Lambda^m\rho-i(\rho) div\Lambda^m \Sigma{(u)}-i(\rho) div\Lambda^m \tau  \\ \notag
&=-u\cdot\nabla\Lambda^m u-[\Lambda^m,u]\nabla u-[\Lambda^m,h(\rho)   -\gamma]\nabla\rho+[\Lambda^m,i(\rho)-1]div\Sigma{(u)}+[\Lambda^m,i(\rho)-1]div~\tau.
\end{align}
Multiplying $h(\rho) \Lambda^s \rho$ to $(\ref{h1})$ and integrating over $\mathbb{R}^{d}$ with $x$, we obtain
\begin{align}\label{co6}
&\frac 1 2 \frac {d} {dt} \int_{\mathbb{R}^{d}}h(\rho)|\Lambda^s \rho|^2 dx+\int_{\mathbb{R}^{d}}P'(1+\rho)\Lambda^s \rho div\Lambda^s u dx =\frac 1 2
\int_{\mathbb{R}^{d}} \partial_t h(\rho) |\Lambda^s \rho|^2 dx  \\ \notag
&-\int_{\mathbb{R}^{d}}\Lambda^s \rho \cdot h(\rho)  u\cdot\nabla \Lambda^s \rho dx-\int_{\mathbb{R}^{d}}[\Lambda^s,u]\nabla\rho\cdot h(\rho) \Lambda^s\rho dx-\int_{\mathbb{R}^{d}}[\Lambda^s,(1+\rho)]div~u\cdot h(\rho) \Lambda^s \rho dx.
\end{align}
H\"{o}lder's inequality yields that
\begin{align*}
\frac 1 2
\int_{\mathbb{R}^{d}} \partial_t h(\rho) |\Lambda^s \rho|^2 dx
\lesssim (\|u\|_{H^s}+\|\rho\|_{H^s})\|\Lambda^s \rho\|^2_{L^2}.
\end{align*}
By virtue of integration by parts, we get
\begin{align*}
-\int_{\mathbb{R}^{d}}\Lambda^s \rho \cdot h(\rho)  u\cdot\nabla \Lambda^s \rho dx
=\frac 1 2 \int_{\mathbb{R}^{d}} div(h(\rho)u)|\Lambda^s \rho|^2 dx
\lesssim (\|u\|_{H^s}+\|\rho\|_{H^s})\|\Lambda^s \rho\|^2_{L^2}.
\end{align*}
By Lemma \ref{Lemma4}, we obtain
\begin{align*}
&-\int_{\mathbb{R}^{d}}[\Lambda^s,u]\nabla\rho\cdot h(\rho) \Lambda^s\rho dx-\int_{\mathbb{R}^{d}}[\Lambda^s,\rho]divu\cdot h(\rho) \Lambda^s \rho dx  \\
&\lesssim (\|\Lambda^{s}u\|_{L^2}\|\nabla\rho\|_{L^\infty}+\|\nabla u\|_{L^\infty}\|\Lambda^{s}\rho\|_{L^2})\|\Lambda^s \rho\|_{L^2} \\
&\lesssim (\|\Lambda^{s}u\|_{L^2}\|\rho\|_{H^s}+\|u\|_{H^s}\|\Lambda^{s} \rho\|_{L^2})\|\Lambda^s \rho\|_{L^2}.
\end{align*}
Multiplying $(1+\rho) \Lambda^s u$ to $(\ref{h2})$ with $m=s$ and integrating over $\mathbb{R}^{d}$ with $x$, then we have
\begin{align}\label{co7}
&\frac 1 2 \frac {d} {dt} \|(1+\rho)^{\frac 1 2}\Lambda^s u\|^2_{L^2}+\int_{\mathbb{R}^{d}}P'(1+\rho)\nabla\Lambda^s \rho \Lambda^s u dx  \\ \notag
&+\mu\|\nabla\Lambda^s u\|^2_{L^2}+(\mu+\mu')\|div\Lambda^s u\|^2_{L^2}-\int_{\mathbb{R}^{d}}div\Lambda^s \tau \Lambda^s u dx  \\ \notag
&=\frac 1 2 \int_{\mathbb{R}^{d}} \partial_t\rho |\Lambda^s u|^2 dx
-\int_{\mathbb{R}^{d}}\Lambda^s u \cdot(1+\rho)u\cdot\nabla \Lambda^s u dx  \\ \notag
&-\int_{\mathbb{R}^{d}}[\Lambda^s,u]\nabla u (1+\rho)\Lambda^s u dx
-\int_{\mathbb{R}^{d}}[\Lambda^s,h(\rho)-\gamma]\nabla\rho (1+\rho)\Lambda^s u dx  \\ \notag
&+\int_{\mathbb{R}^{d}}[\Lambda^s,i(\rho)-1]div\Sigma{(u)} (1+\rho)\Lambda^s u dx
+\int_{\mathbb{R}^{d}}[\Lambda^s,i(\rho)-1]div~\tau (1+\rho)\Lambda^s u dx.
\end{align}
By virtue of Lemmas \ref{Lemma1}, \ref{Lemma2} and \ref{Lemma4}, we obtain
\begin{align*}
\begin{split}
&\frac 1 2 \int_{\mathbb{R}^{d}} \partial_t\rho |\Lambda^s u|^2 dx
-\int_{\mathbb{R}^{d}}\Lambda^s u \cdot(1+\rho)u\cdot\nabla \Lambda^s u dx
-\int_{\mathbb{R}^{d}}[\Lambda^s,u]\nabla u (1+\rho)\Lambda^s u dx  \\ \notag
&\lesssim \|u\|_{H^s}\|\Lambda^s u\|_{L^2}(\|\Lambda^s u\|_{L^2}+\|\nabla\Lambda^s u\|_{L^2}),  \\
&-\int_{\mathbb{R}^{d}}[\Lambda^s,h(\rho)-\gamma]\nabla\rho (1+\rho)\Lambda^s u dx
+\int_{\mathbb{R}^{d}}[\Lambda^s,i(\rho)-1]div\Sigma{(u)} (1+\rho)\Lambda^s u dx  \\ \notag
&\lesssim\|\rho\|_{H^s}\|\Lambda^s u\|_{L^2}(\|\Lambda^{s} \rho\|_{L^2}+\|\nabla^2 u\|_{H^{s-1}}),  \\
&\int_{\mathbb{R}^{d}}[\Lambda^s,i(\rho)-1]div~\tau (1+\rho)\Lambda^s u dx\lesssim\|\rho\|_{H^s}\|\Lambda^s u\|_{L^2}\|\nabla_R \nabla g\|_{H^{s-1}(\mathcal{L}^{2})}.
\end{split}
\end{align*}
Integrating by part, we have
\begin{align*}
-\int_{\mathbb{R}^{d}}P'(1+\rho)(\Lambda^s u\nabla\Lambda^s\rho+\Lambda^s\rho div\Lambda^su)dx
&=\int_{\mathbb{R}^{d}}P''(1+\rho)\Lambda^s\rho \Lambda^s u\nabla\rho dx  \\
&\lesssim \|\rho\|_{H^s}\|\Lambda^s u\|_{L^2}\|\Lambda^s \rho\|_{L^2}.
\end{align*}
Multiplying $\nabla\Lambda^{s-1} \rho$ to $(\ref{h2})$ with $m=s-1$ and integrating over $\mathbb{R}^{d}$ with $x$, then we get
\begin{align}\label{co8}
&\frac {d} {dt} \int_{\mathbb{R}^{d}}\Lambda^{s-1} u \cdot\nabla\Lambda^{s-1} \rho dx
+\gamma\|\nabla\Lambda^{s-1}\rho\|^2_{L^{2}}  \\ \notag
&=-\int_{\mathbb{R}^{d}} \Lambda^{s-1} \rho_t div\Lambda^{s-1} u dx-\int_{\mathbb{R}^{d}}\nabla\Lambda^{s-1} \rho\cdot u\cdot\nabla \Lambda^{s-1} u dx \\ \notag
&-\int_{\mathbb{R}^{d}}[\Lambda^{s-1},u]\nabla u \nabla\Lambda^{s-1} \rho dx-\int_{\mathbb{R}^{d}}\Lambda^{s-1}((h(\rho)-\gamma)\nabla\rho) \nabla\Lambda^{s-1} \rho dx \\ \notag
&+\int_{\mathbb{R}^{d}}\Lambda^{s-1}(i(\rho)div\Sigma{(u)}) \nabla\Lambda^{s-1} \rho dx
+\int_{\mathbb{R}^{d}}\Lambda^{s-1}(i(\rho)div\tau) \nabla\Lambda^{s-1} \rho dx.
\end{align}
Using Lemma \ref{Lemma4} and Lemma \ref{Lemma0}, we can deduce that
\begin{align*}
-\int_{\mathbb{R}^{d}}[\Lambda^{s-1},u]\nabla u \nabla\Lambda^{s-1} \rho dx&\lesssim \|\Lambda^{s} \rho\|_{L^2}\|\nabla u\|_{L^{\infty}}\|\Lambda^{s-1} u\|_{L^2}\\
&\lesssim \|\Lambda^{s} \rho\|_{L^2}\|u\|^{1-\frac 1 s-\frac d {2s}}_{L^{2}}\|\Lambda^{s} u\|^{\frac 1 s+\frac d {2s}}_{L^2}\|u\|^{\frac 1 s}_{L^{2}}\|\Lambda^{s} u\|^{1-\frac 1 s}_{L^2}\\
&\lesssim \|\Lambda^{s} \rho\|_{L^2}\|u\|_{H^{s}}\|\Lambda^{s} u\|_{L^2},
\end{align*}
and
\begin{align*}
&-\int_{\mathbb{R}^{d}} \Lambda^{s-1} \rho_t div\Lambda^{s-1} u dx
-\int_{\mathbb{R}^{d}}\nabla\Lambda^{s-1} \rho\cdot u\cdot\nabla \Lambda^{s-1} u dx \\
&\lesssim \|u\|_{H^s}\|\Lambda^{s} u\|_{L^2}\|\Lambda^{s} \rho\|_{L^2}+\|\Lambda^{s} u\|^2_{L^2}(1+\|\rho\|_{L^{\infty}})\\
&+\|\Lambda^{s} u\|_{L^2}(\|\nabla\rho\|_{L^{\infty}}\|\Lambda^{s-1} u\|_{L^2}+\|\nabla u\|_{L^{\infty}}\|\Lambda^{s-1}\rho\|_{L^2})  \\
&\lesssim \|\Lambda^{s} u\|^2_{L^2}+\|\Lambda^{s} u\|_{L^2}(\|\Lambda^{s} u\|_{L^2}+\|\Lambda^{s}\rho\|_{L^2})(\|\rho\|_{H^s}+\|u\|_{H^s}).
\end{align*}
Similar, we get
\begin{align*}
&-\int_{\mathbb{R}^{d}}\Lambda^{s-1}((h(\rho)-\gamma)\nabla\rho) \nabla\Lambda^{s-1} \rho dx +\int_{\mathbb{R}^{d}}\Lambda^{s-1}(i(\rho)div\Sigma{(u)}) \nabla\Lambda^{s-1} \rho dx\\
&\lesssim \|\Lambda^{s} \rho\|_{L^2}(\|\rho\|_{L^{\infty}}\|\Lambda^{s} \rho\|_{L^2}+\|\nabla\rho\|_{L^{\infty}}\|\Lambda^{s-1} \rho\|_{L^2}+\|\nabla^2 u\|_{H^{s-1}}+\|\nabla^2 u\|_{H^{s-1}}\|\rho\|_{H^{s-1}})\\
&\lesssim \|\Lambda^{s} \rho\|_{L^2}(\|\rho\|_{H^{s}}\|\Lambda^{s} \rho\|_{L^2}+\|\nabla^2 u\|_{H^{s-1}}+\|\nabla^2 u\|_{H^{s-1}}\|\rho\|_{H^{s-1}}).
\end{align*}
Using Lemma \ref{Lemma1} and Lemma \ref{Lemma2}, we have
\begin{align*}
\int_{\mathbb{R}^{d}}\Lambda^{s-1}(i(\rho)div\tau) \nabla\Lambda^{s-1} \rho dx\lesssim \|\Lambda^{s} \rho\|_{L^2}\|\nabla_R \nabla g\|_{H^{s-1}(\mathcal{L}^{2})}(\|\rho\|_{H^{s-1}}+1).
\end{align*}

Combining \eqref{co5} and the estimates for \eqref{co6}-\eqref{co8}, we obtain the higher order derivatives estimates for $\eqref{eq1}$:
\begin{align}\label{high estimate}
&\frac {d} {dt} (\|h(\rho)^{\frac 1 2}\Lambda^s\rho \|^2_{L^2}+\|(1+\rho)^{\frac 1 2}\Lambda^s u\|^2_{L^2}+\|\Lambda^s g\|^2_{L^2(\mathcal{L}^{2})}+2\eta\int_{\mathbb{R}^{d}} \Lambda^{s-1} u\nabla\Lambda^{s-1} \rho dx)  \\ \notag
&+2(\mu\|\nabla \Lambda^s u\|^2_{L^2}+(\mu+\mu')\|div\Lambda^s u\|^2_{L^2}+
\eta\gamma\|\nabla\Lambda^{s-1}\rho\|^2_{L^2}+\|\nabla_R \Lambda^s g\|^2_{L^2(\mathcal{L}^{2})})  \\ \notag
&\lesssim \|\rho\|_{H^s}\|\nabla^2 u\|_{H^{s-1}}\|\nabla\nabla_R g\|_{H^{s-1}(\mathcal{L}^{2})}+(\|\rho\|_{H^{s}}+\|u\|_{H^{s}})(\|\nabla^2 \rho\|^2_{H^{s-2}}+\|\nabla^2 u\|^2_{H^{s-1}})  \\ \notag
&+\|g\|_{L^\infty(\mathcal{L}^{2})}\|\nabla \Lambda^s u\|_{L^2}\|\nabla_R\Lambda^s g\|_{L^2(\mathcal{L}^{2})}+\|u\|_{H^s}\|\nabla_R\nabla g\|^2_{H^{s-1}(\mathcal{L}^{2})} \\ \notag
&+\eta(\|\nabla^2 u\|^2_{H^{s-1}}+\|\nabla^2 u\|_{H^{s-1}}\|\nabla^2 \rho\|_{H^{s-2}})  +\eta(\|\rho\|_{H^s}+\|u\|_{H^s})(\|\nabla^2 \rho\|^2_{H^{s-2}}+\|\nabla^2 u\|^2_{H^{s-1}})\\ \notag
&+\eta\|\nabla^2 \rho\|_{H^{s-2}}\|\nabla\nabla_R g\|_{H^{s-1}(\mathcal{L}^{2})}(1+\|\rho\|_{H^{s}}).
\end{align}
Note that the precise estimate \eqref{high estimate} will be used in the next section again.

Choosing sufficiently small constant $\eta>0$, we have $E(t)\sim E_\eta(t)$ and $D(t)\sim D_\eta(t)$.
Then combining \eqref{low estimate} and \eqref{high estimate} with sufficiently small $\eta>0$ and $\epsilon$, we finally deduce that
$$\frac d {dt} E_\eta(t)+D_\eta(t)\leq 0,$$
which implies that there exist $C_0>1$ such that
\begin{align*}
\sup_{t\in[0,T]} E(t)+\int_{0}^{T}D(t)dt\leq C_{0}E(0).
\end{align*}
We thus complete the proof of Proposition \ref{pro2}.
\end{proof}

{\bf The proof of Theorem \ref{th1}:}

Applying Proposition \ref{pro1} and Proposition \ref{pro2}, we prove the global-in-time solutions of the compressible polymeric system \eqref{eq1} by using the standard continuum argument. Let $E(0)\leq \epsilon_0$ with $\epsilon_0=\frac {\epsilon} {2C_0}$ and $C_0>1$. Applying Proposition \ref{pro1}, we get the unique local solution result on the time interval $t\in[0,T]$ with $T>0$, satisfying $\sup_{t\in[0,T]} E(t)\leq \epsilon$. Then the global priori estimate in Proposition \ref{pro2} yields
$$E(T)\leq C_0 E(0)\leq \frac {\epsilon} {2}. $$
Applying Proposition \ref{pro1} again, we have the unique local solution result $t\in[T,2T]$, satisfying $\sup_{t\in[T,2T]} E(t)\leq \epsilon$. So it holds that
$$\sup_{t\in[0,2T]} E(t)\leq \epsilon.$$
Then the global priori estimate in Proposition \ref{pro2} yields
$$E(2T)\leq C_0 E(0)\leq \frac {\epsilon} {2}. $$
Repeating this bootstrap argument, we prove the global existence of strong solution of the compressible polymeric system \eqref{eq1}. Moreover, we obtain $\sup_{t\in[0,\infty)} E(t)+\int_{0}^{\infty}D(t)dt\leq C_{0}E(0).$
\hfill$\Box$

\section{The $L^2$ decay rate}
We investigate the long time behaviour for the compressible FENE dumbbell model in this section. The first difficult for us is that the stress tensor $\tau$ does not decay fast enough. Therefore, we failed to use the bootstrap argument as in \cite{Schonbek1985,Luo-Yin,2021Global}. To overcome this difficulty, we need to consider the coupling effect between $\rho$, $u$ and $g$ in Fourier space. Motivated by \cite{He2009} and \cite{2018Global}, we obtain the $L^2$ decay rate by taking Fourier transform in \eqref{eq1} and using the Fourier splitting method. By virtue of the standard method, one can not obtain the optimal decay rate. However, we can obtain a weaker result as follow.
\begin{prop}\label{pro3}
Let $(\rho_0,u_0,g_0)$ satisfy the same condition in Theorem \ref{th2}. For any $t\in(0,+\infty)$, we have
\begin{align}
E(t) \leq C(1+t)^{-\frac d 4}.
\end{align}
\end{prop}
\begin{proof}
Firstly, we consider the simple case $d\geq 5$.
According to Proposition \ref{pro2}, we have
\begin{align}\label{energy estimate}
\frac d {dt} E_\eta(t)+D_\eta(t)\leq 0.
\end{align}
Using Schonbek's strategy, we consider $S(t)=\{\xi:|\xi|^2\leq C_d(1+t)^{-1}\}$ where the constant $C_d$ will be chosen later on.  Then we have
$$\|\nabla u\|^2_{H^s}=\int_{S(t)}(1+|\xi|^{2s})|\xi|^2|\hat{u}(\xi)|^2 d\xi+\int_{S(t)^c}(1+|\xi|^{2s})|\xi|^2|\hat{u}(\xi)|^2 d\xi.$$
We deduce that
$$\frac {C_d} {1+t} \int_{S(t)^c}(1+|\xi|^{2s})|\hat{u}(\xi)|^2 d\xi\leq\|\nabla u\|^2_{H^s},$$
and
$$\frac {C_d} {1+t} \int_{S(t)^c}(1+|\xi|^{2s-2})|\hat{\rho}(\xi)|^2 d\xi\leq\|\nabla \rho\|^2_{H^{s-1}},$$ which implies that
\begin{align}\label{ineq1}
\frac d {dt} E_\eta(t)+\frac {\mu C_d} {1+t}\| u\|^2_{H^s}+\frac {\eta\gamma C_d} {1+t}\|\rho\|^2_{H^{s-1}}
+\|\nabla_R g\|^2_{H^{s}(\mathcal{L}^{2})}\leq \frac {CC_d} {1+t}\int_{S(t)}|\hat{u}(\xi)|^2+|\hat{\rho}(\xi)|^2 d\xi.
\end{align}

From now on, we consider the $L^2$ estimate to the low frequency part of $(\rho,u)$. Taking Fourier transform with respect to $x$ in \eqref{eq1}, we obtain
\begin{align}\label{eq2}
\left\{
\begin{array}{ll}
\hat{\rho}_t+i\xi_{k} \hat{u}^k=\hat{F},  \\[1ex]
\hat{u}^{j}_t+\mu|\xi|^2 \hat{u}^j+(\mu+\mu')\xi_{j} \xi_{k} \hat{u}^k+i\xi_{j} \gamma\hat{\rho}-i\xi_{k} \hat{\tau}^{jk}=\hat{G}^j,  \\[1ex]
\hat{g}_t+\mathcal{L}\hat{g}-i\xi_{k} \hat{u}^j R_j \partial_{R_k}\mathcal{U}+i\xi_{k} \hat{u}^k=\hat{H}, \\[1ex]
\end{array}
\right.
\end{align}
where $F=-div(\rho u)$, $G=-u\cdot\nabla u+[i(\rho)-1](div\Sigma(u)+div \tau)+[\gamma-h(\rho)]\nabla\rho$ and $H=-u\cdot\nabla g-\frac 1 {\psi_\infty} \nabla_R\cdot(\nabla u Rg\psi_\infty)$.  \\
Multiplying $\bar{\hat{\rho}}(t,\xi)$ to the first equation of \eqref{eq2} and taking the real part, we have
\begin{align}\label{eq3}
\frac 1 2 \frac d {dt} |\hat{\rho}|^2+\mathcal{R}e[i\xi\cdot\hat{u}\bar{\hat{\rho}}]=\mathcal{R}e[\hat{F}\bar{\hat{\rho}}].
\end{align}
Multiplying $\bar{\hat{u}}^j(t,\xi)$ with $1\leq j\leq d$ to the second equation of \eqref{eq2} and considering the real part, we deduce that
\begin{align}\label{eq4}
\frac 1 2 \frac d {dt} |\hat{u}|^2+\mathcal{R}e[\gamma\hat{\rho}i\xi\cdot\bar{\hat{u}}]+\mu|\xi|^2 |\hat{u}|^2+(\mu+\mu')|\xi\cdot\hat{u}|^2-\mathcal{R}e[i\xi\otimes\bar{\hat{u}}(t,\xi):\hat{\tau}]=\mathcal{R}e[\hat{G}\cdot\bar{\hat{u}}].
\end{align}
Multiplying $\bar{\hat{g}}(t,\xi,R)\psi_\infty$ to the third equation of \eqref{eq2}, integrating over $B$ with $R$ and taking the real part, we have
\begin{align}\label{eq5}
\frac 1 2 \frac d {dt} \|\hat{g}\|^2_{\mathcal{L}^2}+
\|\nabla_R \hat{g}\|^2_{\mathcal{L}^2}-\mathcal{R}e[i\xi\otimes\hat{u}:\bar{\hat{\tau}}]=\mathcal{R}e[\int_{B}\hat{H}\bar{\hat{g}}\psi_\infty dR],
\end{align}
where we using the fact $\int_{B}i\xi_{k} \hat{u}^k\bar{\hat{g}}\psi_\infty dR=0$.
It is easy to verify that
$$\mathcal{R}e[i\xi\cdot\hat{u}\bar{\hat{\rho}}]+\mathcal{R}e[\hat{\rho}i\xi\cdot\bar{\hat{u}}]
=\mathcal{R}e[i\xi\otimes\bar{\hat{u}}(t,\xi):\hat{\tau}]+\mathcal{R}e[i\xi\otimes\hat{u}:\bar{\hat{\tau}}]=0,$$
which implies that
\begin{align}\label{eq6}
&\frac 1 2 \frac d {dt} (\gamma|\hat{\rho}|^2+|\hat{u}|^2+\|\hat{g}\|^2_{\mathcal{L}^2})+\mu|\xi|^2 |\hat{u}|^2+(\mu+\mu')|\xi\cdot\hat{u}|^2
+\|\nabla_R \hat{g}\|^2_{\mathcal{L}^2}  \\ \notag
&=\mathcal{R}e[\gamma\hat{F}\bar{\hat{\rho}}]+\mathcal{R}e[\hat{G}\cdot\bar{\hat{u}}]+\mathcal{R}e[\int_{B}\hat{H}\bar{\hat{g}}\psi_\infty dR].
\end{align}
Multiplying $i\xi\cdot\bar{\hat{u}}$ to the first equation of \eqref{eq2} and considering the real part, we obtain
\begin{align}\label{eq7}
\mathcal{R}e[\hat{\rho}_t i\xi\cdot\bar{\hat{u}}]-|\xi\cdot\hat{u}|^2=\mathcal{R}e[\hat{F}i\xi\cdot\bar{\hat{u}}].
\end{align}
Multiplying $-i\xi_j\bar{\hat{\rho}}$ with $1\leq j\leq d$ to the second equation of \eqref{eq2} and taking the real part, we get
\begin{align}\label{eq8}
\mathcal{R}e[\hat{\rho}i\xi\cdot\bar{\hat{u}}_t]+\gamma|\xi|^2 |\hat{\rho}|^2+(2\mu+\mu')|\xi|^2 \mathcal{R}e[\hat{\rho}i\xi\cdot\bar{\hat{u }}] -\mathcal{R}e[\hat{\rho}\xi\otimes\xi:\bar{\hat{\tau}}]=\mathcal{R}e[\bar{\hat{G}}\cdot i\xi\hat{\rho}].
\end{align}
It follows from \eqref{eq6}$-$\eqref{eq8} that
\begin{align}\label{eq9}
&\frac 1 2 \frac d {dt} (\gamma|\hat{\rho}|^2+|\hat{u}|^2+\|\hat{g}\|^2_{\mathcal{L}^2}+2(\mu+\mu')\mathcal{R}e[\hat{\rho}i\xi\cdot\bar{\hat{u}}])+\mu|\xi|^2 |\hat{u}|^2+(\mu+\mu')\gamma|\xi|^2 |\hat{\rho}|^2
+\|\nabla_R \hat{g}\|^2_{\mathcal{L}^2}  \\ \notag
&=-(\mu+\mu')(2\mu+\mu')|\xi|^2 \mathcal{R}e[\hat{\rho}i\xi\cdot\bar{\hat{u }}] +(\mu+\mu')\mathcal{R}e[\hat{\rho}\xi\otimes\xi:\bar{\hat{\tau}}]+(\mu+\mu')\mathcal{R}e[\hat{F}i\xi\cdot\bar{\hat{u}}]  \\ \notag
&+(\mu+\mu')\mathcal{R}e[\bar{\hat{G}}\cdot i\xi\hat{\rho}]+\mathcal{R}e[\gamma\hat{F}\bar{\hat{\rho}}]+\mathcal{R}e[\hat{G}\cdot\bar{\hat{u}}]+\mathcal{R}e[\int_{B}\hat{H}\bar{\hat{g}}\psi_\infty dR].
\end{align}
Consider $\xi\in S(t)$, we deduce that
\begin{align}\label{ineq2}
&(\mu+\mu')\mathcal{R}e[\hat{F}i\xi\cdot\bar{\hat{u}}]+(\mu+\mu')\mathcal{R}e[\bar{\hat{G}}\cdot i\xi\hat{\rho}]+\mathcal{R}e[\gamma\hat{F}\bar{\hat{\rho}}] \\
\notag
&\leq C(|\widehat{\rho u}|^2+|\hat{G}|^2)+\frac 1 {10} (\mu|\xi|^2 |\hat{u}|^2+(\mu+\mu')\gamma|\xi|^2 |\hat{\rho}|^2).
\end{align}
Let $t$ be sufficiently large, we obtain
\begin{align}\label{ineq3}
2(\mu+\mu')\mathcal{R}e[\hat{\rho}i\xi\cdot\bar{\hat{u}}]\leq
\frac 1 {10}(|\hat{u}|^2+\gamma|\hat{\rho}|^2).
\end{align}
Integrating by part and using Lemmas \ref{Lemma1}, \ref{Lemma2}, we get
\begin{align*}
&-(\mu+\mu')(2\mu+\mu')|\xi|^2 \mathcal{R}e[\hat{\rho}i\xi\cdot\bar{\hat{u }}]+(\mu+\mu')\mathcal{R}e[\hat{\rho}\xi\otimes\xi:\bar{\hat{\tau}}]  \\
&\leq\frac 1 {10}(\mu|\xi|^2 |\hat{u}|^2+(\mu+\mu')\gamma|\xi|^2 |\hat{\rho}|^2
+\|\nabla_R \hat{g}\|^2_{\mathcal{L}^2}),
\end{align*}
and
\begin{align*}
\begin{split}
\mathcal{R}e[\int_{B}\hat{H}\bar{\hat{g}}\psi_\infty dR]&\leq C_\delta\int_{B}\psi_\infty|\mathcal{F}(u\cdot\nabla g)|^2
+\psi_\infty|\mathcal{F}(\nabla u\cdot{R}g)|^2  dR+\delta\|\nabla_R \hat{g}\|^2_{\mathcal{L}^2},  \\
|\xi|^2\|\hat{g}\|^2_{\mathcal{L}^2}&\leq\|\nabla_R \hat{g}\|^2_{\mathcal{L}^2}.
\end{split}
\end{align*}
Combining all the estimates for \eqref{eq9}, we deduce that
\begin{align}\label{ineq4}
&|\hat{\rho}|^2+|\hat{u}|^2+\|\hat{g}\|^2_{\mathcal{L}^2}
\leq C(|\hat{\rho}_0|^2+|\hat{u}_0|^2+\|\hat{g}_0\|^2_{\mathcal{L}^2})+C\int_{0}^{t}|\hat{G}\cdot\bar{\hat{u}}|+|\widehat{\rho u}|^2+|\hat{G}|^2 ds  \\ \notag
&+C_\delta\int_{0}^{t}\int_{B}\psi_\infty|\mathcal{F}(u\cdot\nabla g)|^2+\psi_\infty|\mathcal{F}(\nabla u\cdot{R}g)|^2 dRds.
\end{align}
Integrating over $S(t)$ with $\xi$, then we have
\begin{align}\label{ineq5}
&\int_{S(t)}|\hat{\rho}|^2+|\hat{u}|^2+\|\hat{g}\|^2_{\mathcal{L}^2}d\xi
\leq C\int_{S(t)} (|\hat{\rho}_0|^2+|\hat{u}_0|^2+\|\hat{g}_0\|^2_{\mathcal{L}^2})d\xi+C\int_{S(t)}\int_{0}^{t}|\hat{G}\cdot\bar{\hat{u}}|+|\widehat{\rho u}|^2+|\hat{G}|^2dsd\xi  \\ \notag
&+C_\delta\int_{S(t)}\int_{0}^{t}\int_{B}\psi_\infty|\mathcal{F}(u\cdot\nabla g)|^2+\psi_\infty|\mathcal{F}(\nabla u\cdot{R}g)|^2 dRdsd\xi.
\end{align}
If $E(0)<\infty$ and $(\rho_0,u_0,g_0)\in \dot{B}^{-\frac d 2}_{2,\infty}\times \dot{B}^{-\frac d 2}_{2,\infty}\times \dot{B}^{-\frac d 2}_{2,\infty}(\mathcal{L}^2)$, using Proposition \ref{pro0}, we have
\begin{align}\label{ineq6}
\int_{S(t)}(|\hat{\rho}_0|^2+|\hat{u}_0|^2+\|\hat{g}_0\|^2_{\mathcal{L}^2})d\xi
&\leq\sum_{j\leq \log_2[\frac {4} {3}C_d^{\frac 1 2 }(1+t)^{-\frac 1 2}]}\int_{\mathbb{R}^{d}} 2\varphi^2(2^{-j}\xi)(|\hat{\rho}_0|^2+|\hat{u}_0|^2+\|\hat{g}_0\|^2_{\mathcal{L}^2})d\xi \\ \notag
&\leq\sum_{j\leq \log_2[\frac {4} {3}C_d^{\frac 1 2 }(1+t)^{-\frac 1 2}]}(\|\dot{\Delta}_j u_0\|^2_{L^2}+\|\dot{\Delta}_j \rho_0\|^2_{L^2}+\|\dot{\Delta}_j g_0\|^2_{L^2(\mathcal{L}^2)}) \\ \notag
&\leq\sum_{j\leq \log_2[\frac {4} {3}C_d^{\frac 1 2 }(1+t)^{-\frac 1 2}]}2^{jd}(\|u_0\|^2_{\dot{B}^{-\frac d 2}_{2,\infty}}+\|\rho_0\|^2_{\dot{B}^{-\frac d 2}_{2,\infty}}+\|g_0\|^2_{\dot{B}^{-\frac d 2}_{2,\infty}(\mathcal{L}^2)}) \\ \notag
&\leq C(1+t)^{-\frac d 2}(\|u_0\|^2_{\dot{B}^{-\frac d 2}_{2,\infty}}+\|\rho_0\|^2_{\dot{B}^{-\frac d 2}_{2,\infty}}+\|g_0\|^2_{\dot{B}^{-\frac d 2}_{2,\infty}(\mathcal{L}^2)}).
\end{align}
Using Minkowski's inequality and Theorem \ref{th1}, we obtain
\begin{align}\label{ineq7}
\int_{S(t)}\int_{0}^{t}|\widehat{\rho u}|^2dsd\xi
&=\int_{0}^{t}\int_{S(t)}|\widehat{\rho u}|^2 d\xi ds  \\ \notag
&\leq C\int_{S(t)}d\xi \int_{0}^{t}\||\widehat{\rho u}|^2\|_{L^{\infty}}ds \\ \notag
&\leq C(1+t)^{-\frac d 2} \int_{0}^{t}\|\rho\|^2_{L^{2}}\|u\|^2_{L^{2}}ds \\ \notag
&\leq C(1+t)^{-\frac d 2+1},
\end{align}
and
\begin{align*}
\int_{S(t)}\int_{0}^{t}|\hat{G}|^2dsd\xi
&\leq C\int_{S(t)}d\xi \int_{0}^{t}\||\hat{G}|^2\|_{L^{\infty}}ds \\ \notag
&\leq C(1+t)^{-\frac d 2}.
\end{align*}
Similarly, we deduce that
\begin{align}\label{ineq8}
\int_{S(t)}\int_{0}^{t}|\hat{G}\cdot\bar{\hat{u}}|dsd\xi
&=\int_{0}^{t}\int_{S(t)}|\hat{G}\cdot\bar{\hat{u}}| d\xi ds  \\ \notag
&\leq C(\int_{S(t)}d\xi)^{\frac 1 2} \int_{0}^{t}\|\hat{G}\cdot\bar{\hat{u}}\|_{L^{2}}ds \\ \notag
&\leq C(1+t)^{-\frac d 4} \int_{0}^{t}(\|u\|^2_{L^{2}}+\|\rho\|^2_{L^{2}})D(s)^{\frac 1 2}ds \\ \notag
&\leq C(1+t)^{-\frac d 4+\frac 1 2}.
\end{align}
Using Theorem \ref{th1} and Lemma \ref{Lemma1}, we get
\begin{align*}
&\int_{S(t)}\int_{0}^{t}\int_{B}\psi_\infty|\mathcal{F}(u\cdot\nabla g)|^2+\psi_\infty|\mathcal{F}(\nabla u\cdot{R}g)|^2dRdsd\xi  \\ \notag
&\leq C(1+t)^{-\frac d 2} \int_{0}^{t}\|u\|^2_{L^{2}}\|\nabla g\|^2_{L^{2}(\mathcal{L}^{2})}+\|\nabla u\|^2_{L^{2}}\|g\|^2_{L^{2}(\mathcal{L}^{2})}ds \\ \notag
&\leq C(1+t)^{-\frac d 2}.
\end{align*}
Plugging the above estimates into \eqref{ineq5}, we obtain
\begin{align}\label{ineq9}
\int_{S(t)}|\hat{\rho}(t,\xi)|^2+|\hat{u}(t,\xi)|^2 d\xi\leq C(1+t)^{-\frac d 4+\frac 1 2}.
\end{align}
According to \eqref{ineq1} and \eqref{ineq9}, we deduce that
\begin{align*}
\frac d {dt} E_\eta(t)+\frac {\mu C_d} {1+t}\| u\|^2_{H^s}+\frac {\eta\gamma C_d} {1+t}\|\rho\|^2_{H^{s-1}}
+\|\nabla_R g\|^2_{H^{s}(\mathcal{L}^{2})}\leq \frac {CC_d} {1+t}(1+t)^{-\frac d 4+\frac 1 2},
\end{align*}
If $C_d$ large enough, according to \eqref{energy estimate}, then we have
\begin{align}\label{ineq10}
(1+t)^{\frac d 4+\frac 1 2}E_\eta(t)&\leq C(1+t)+C\int_{0}^{t}\|\Lambda^s \rho\|^2_{L^{2}}(1+s)^{\frac d 4-\frac 1 2}ds  \\ \notag
&\leq C(1+t)+C\int_{0}^{t}D_\eta(s)(1+s)^{\frac d 4-\frac 1 2}ds  \\  \notag
&\leq C(1+t)+C\int_{0}^{t}E_\eta(s)(1+s)^{\frac d 4-\frac 3 2}ds\\  \notag
&\leq C(1+t)+CE_\eta(t)(1+t)^{\frac d 4-\frac 1 2},
\end{align}
which implies that
\begin{align}\label{ineq11}
E_\eta(t)\leq C(1+t)^{-\frac d 4+\frac 1 2}.
\end{align}
We now improve the decay rate in \eqref{ineq11} by estimating \eqref{ineq7} and \eqref{ineq8} again. Since $d\geq 5$, it follows that
\begin{align*}
\int_{S(t)}\int_{0}^{t}|\widehat{\rho u}|^2dsd\xi
&\leq C(1+t)^{-\frac d 2} \int_{0}^{t}\|\rho\|^2_{L^{2}}\|u\|^2_{L^{2}}ds \\ \notag
&\leq C(1+t)^{-\frac d 4},
\end{align*}
and
\begin{align*}
\int_{S(t)}\int_{0}^{t}|\hat{G}\cdot\bar{\hat{u}}|dsd\xi
&\leq C(1+t)^{-\frac d 4} \int_{0}^{t}(\|u\|^2_{L^{2}}+\|\rho\|^2_{L^{2}})D(s)^{\frac 1 2}ds \\ \notag
&\leq C(1+t)^{-\frac d 4},
\end{align*}
which implies that
\begin{align}\label{ineq12}
\frac d {dt} E_\eta(t)+\frac {\mu C_d} {1+t}\| u\|^2_{H^s}+\frac {\eta\gamma C_d} {1+t}\|\rho\|^2_{H^{s-1}}
+\|\nabla_R g\|^2_{H^{s}(\mathcal{L}^{2})}\leq \frac {CC_d} {1+t}(1+t)^{-\frac d 4}.
\end{align}
Then the proof of \eqref{ineq11} implies that
\begin{align}\label{ineq13}
E(t)\leq CE_\eta(t)\leq C(1+t)^{-\frac d 4}.
\end{align}

For $d=3,4$, we can not obtain the optimal decay rate $(1+t)^{-\frac d 4}$ directly. Indeed, we can first prove that $E(t)\leq (1+t)^{-\frac{d}{4}+\frac{1}{2}}$. By using the standard bootstrap argument, one can improve the decay rate to $(1+t)^{-\frac d 4}$.  We omit the proof here.
\end{proof}

\begin{rema}
The proposition \ref{pro3} indicates that
$$\|\rho\|_{L^2}+\|u\|_{L^2}\leq C(1+t)^{-\frac d 8}.$$
Combining with the incompressible FENE model and CNS system, one can see that this is not the optimal time decay.
\end{rema}

In order to improve the decay rate, we have to estimate the high order energy.
Denote that
\begin{align*}
E^1_\eta(t)&=\sum_{n=1,s}(\|h(\rho)^{\frac 1 2}\Lambda^n\rho \|^2_{L^2}+\|(1+\rho)^{\frac 1 2}\Lambda^n u\|^2_{L^2})  \\ \notag
&+\|\Lambda^1 g\|^2_{H^{s-1}(\mathcal{L}^{2})}+2\eta\sum_{m=1,s-1}\int_{\mathbb{R}^{d}} \Lambda^m u\nabla\Lambda^m \rho dx,
\end{align*}
and
$$D^1_\eta (t)=\eta\gamma\|\nabla\Lambda^1 \rho\|^2_{H^{s-2}}+\mu\|\nabla\Lambda^1 u\|^2_{H^{s-1}}+(\mu+\mu')\|div\Lambda^1u\|^2_{H^{s-1}}+\|\Lambda^1\nabla_R g\|^2_{H^{s-1}(\mathcal{L}^{2})}.$$
The following proposition is about the high order energy estimate.
\begin{prop}\label{pro4}
Under the condition in Theorem \ref{th2}, if $t\in(0,+\infty)$, then we have
\begin{align}
\frac d {dt}E^1_\eta (t)+D^1_\eta (t)\leq 0, \quad \quad \text{and}\quad \quad &E^1_\eta \leq C(1+t)^{-\frac d 4-1}.
\end{align}
\end{prop}
\begin{proof}
Applying $\Lambda^1$ to $(\ref{eq1})_3$, we obtain
\begin{align}\label{g1}
&\partial_t\Lambda^1 g+\mathcal{L}\Lambda^1 g+div\Lambda^{1}u+\nabla\Lambda^{1}uR\nabla_R \mathcal{U} \\ \notag
&=-u\cdot\nabla\Lambda^1 g-\Lambda^1 u\cdot\nabla g-\frac 1 {\psi_\infty}\nabla_R \cdot(\Lambda^1\nabla uRg\psi_\infty+R\psi_\infty \Lambda^1 g\nabla u).
\end{align}
Taking the $L^2(\mathcal{L}^{2})$ inner product with $\Lambda^1 g $ to $(\ref{g1})$, then we have
\begin{align}
&\frac {1} {2}\frac {d} {dt} \|\Lambda^1 g\|^2_{L^2(\mathcal{L}^{2})}+\|\nabla_R  \Lambda^1 g\|^2_{L^2(\mathcal{L}^{2})}-\int_{\mathbb{R}^{d}}\nabla\Lambda^{1}u:\Lambda^{1}\tau dx
=-\langle u\cdot\nabla \Lambda^1 g,\Lambda^1 g\rangle  \\ \notag
&-\langle\Lambda^1 u\nabla g,\Lambda^1 g\rangle
-\langle\frac 1 {\psi_\infty} \nabla_R\cdot(\Lambda^1\nabla uRg\psi_\infty),\Lambda^1 g\rangle
-\langle\frac 1 {\psi_\infty} \nabla_R\cdot(R\psi_\infty\Lambda^1 g\nabla u),\Lambda^1 g\rangle.
\end{align}
Integrating by parts, we deduce that
\begin{align*}
\begin{split}
-\langle u\cdot\nabla \Lambda^1 g,\Lambda^1 g\rangle-\langle\Lambda^1 u\nabla g,\Lambda^1 g\rangle &\lesssim \|u\|_{H^s}\|\nabla g\|^2_{H^{1}(\mathcal{L}^{2})},  \\
\langle\frac 1 {\psi_\infty} \nabla_R\cdot(\Lambda^1\nabla uRg\psi_\infty), \Lambda^1 g \rangle
&=-\int_{\mathbb{R}^{d}}\int_{B}(\Lambda^1\nabla uR\psi_\infty g)\nabla_R \Lambda^1 g dRdx \\
&\lesssim \|g\|_{L^\infty(\mathcal{L}^{2})}\|\nabla \Lambda^1 u\|_{L^2}\|\nabla_R\Lambda^1 g\|_{L^2(\mathcal{L}^{2})},  \\
-\langle\frac 1 {\psi_\infty} \nabla_R\cdot(R\psi_\infty\Lambda^1 g\nabla u),\Lambda^s g\rangle
&=\langle R\Lambda^1 g\nabla u,\nabla_R\Lambda^1 g\rangle  \\
&\lesssim \|\nabla_R \Lambda^1 g\|_{L^2(\mathcal{L}^{2})}\|u\|_{H^s}\|\nabla g\|_{L^{2}(\mathcal{L}^{2})},
\end{split}
\end{align*}
which implies that
\begin{multline}\label{g2}
\frac {1} {2}\frac {d} {dt} \|\Lambda^1 g\|^2_{L^2(\mathcal{L}^{2})}+\|\nabla_R \Lambda^1 g\|^2_{L^2(\mathcal{L}^{2})}-\int_{\mathbb{R}^{d}}\nabla\Lambda^{1}u:\Lambda^{1}\tau dx \\
\lesssim \|g\|_{L^\infty(\mathcal{L}^{2})}\|\nabla \Lambda^1 u\|_{L^2}\|\nabla_R\Lambda^1 g\|_{L^2(\mathcal{L}^{2})}+\|u\|_{H^s}\|\nabla_R\Lambda^1 g\|^2_{L^{2}(\mathcal{L}^{2})}.
\end{multline}
Applying $\Lambda^1$ to $(\ref{eq1})_1$ and applying $\Lambda^1$ to $(\ref{eq1})_2$, we obtain
\begin{align}\label{r1}
\partial_t\Lambda^1 \rho+div \Lambda^1 u(1+\rho)
=-u\cdot\nabla\Lambda^1 \rho-\Lambda^1 u\nabla\rho-\Lambda^1 \rho ~div~u,
\end{align}
and
\begin{align}\label{r2}
&\partial_t\Lambda^1 u+h(\rho)\nabla\Lambda^1\rho-i(\rho) div\Lambda^1 \Sigma{(u)}-i(\rho) div\Lambda^1 \tau  \\ \notag
&=-u\cdot\nabla\Lambda^1 u-\Lambda^1 u\nabla u-\Lambda^1 [h(\rho)-\gamma]\nabla\rho+\Lambda^1[i(\rho)-1]div\Sigma{(u)}+\Lambda^1[i(\rho)-1]div~\tau.
\end{align}
Multiplying $h(\rho) \Lambda^1 \rho$ to $(\ref{r1})$ and integrating over $\mathbb{R}^{d}$ with $x$, then we have
\begin{align}\label{g3}
&\frac 1 2 \frac {d} {dt} \int_{\mathbb{R}^{d}}h(\rho)|\Lambda^1 \rho|^2 dx+\int_{\mathbb{R}^{d}}P'(1+\rho)\Lambda^1 \rho div\Lambda^1 u dx =\frac 1 2
\int_{\mathbb{R}^{d}} \partial_t h(\rho) |\Lambda^1 \rho|^2 dx  \\ \notag
&-\int_{\mathbb{R}^{d}}\Lambda^1 \rho \cdot h(\rho)  u\cdot\nabla \Lambda^1 \rho dx-\int_{\mathbb{R}^{d}}\Lambda^1 u\nabla\rho\cdot h(\rho) \Lambda^1\rho dx-\int_{\mathbb{R}^{d}}\Lambda^1 \rho div~u\cdot h(\rho) \Lambda^1 \rho dx.
\end{align}
If $d=3$, we treat with the first term as follow
\begin{align*}
&\frac 1 2\int_{\mathbb{R}^{d}} \partial_t h(\rho) |\Lambda^1 \rho|^2 dx
\lesssim (\|\nabla u\|_{L^3}+\|\nabla\rho\|_{L^3})\|\Lambda^1 \rho\|^2_{L^3} \\
&\lesssim \|u\|^{\frac 1 3}_{H^2}\|\Lambda^2 u\|^{\frac 2 3}_{L^2}\|\rho\|^{\frac 2 3}_{H^2}\|\Lambda^2 \rho\|^{\frac 4 3}_{L^2}+\|\rho\|_{H^2}\|\Lambda^2 \rho\|^2_{L^2}\\
&\lesssim (\|u\|_{H^s}+\|\rho\|_{H^s})(\|\Lambda^2 \rho\|^2_{L^2}+\|\Lambda^2 u\|^2_{L^2}).
\end{align*}
If $d\geq 4$, we treat with the first term as follow
\begin{align*}
&\frac 1 2\int_{\mathbb{R}^{d}} \partial_t h(\rho) |\Lambda^1 \rho|^2 dx
\lesssim (\|\nabla u\|_{L^{\frac d 2}}+\|\nabla\rho\|_{L^{\frac d 2}})\|\Lambda^1 \rho\|^2_{L^{\frac {2d} {d-2}}} \\
&\lesssim (\|u\|_{H^s}+\|\rho\|_{H^s})\|\Lambda^2 \rho\|^2_{L^2}.
\end{align*}
Similarly, we obtain
\begin{align*}
-\int_{\mathbb{R}^{d}}\Lambda^1 u\nabla\rho\cdot h(\rho) \Lambda^1\rho dx-\int_{\mathbb{R}^{d}}\Lambda^1 \rho~ div~u\cdot h(\rho) \Lambda^1 \rho dx  \lesssim (\|u\|_{H^s}+\|\rho\|_{H^s})(\|\Lambda^2 \rho\|^2_{L^2}+\|\Lambda^2 u\|^2_{L^2}).
\end{align*}
Applying integration by parts, we have
\begin{align*}
-\int_{\mathbb{R}^{d}}\Lambda^1 \rho \cdot h(\rho)  u\cdot\nabla \Lambda^1 \rho dx
=\frac 1 2 \int_{\mathbb{R}^{d}} div(h(\rho)u)|\Lambda^1 \rho|^2 dx
\lesssim (\|u\|_{H^s}+\|\rho\|_{H^s})(\|\Lambda^2 \rho\|^2_{L^2}+\|\Lambda^2 u\|^2_{L^2}).
\end{align*}
Multiplying $(1+\rho) \Lambda^1 u$ to $(\ref{r2})$ and integrating over $\mathbb{R}^{d}$ with $x$, we have
\begin{align}\label{g4}
&\frac 1 2 \frac {d} {dt} \|(1+\rho)^{\frac 1 2}\Lambda^1 u\|^2_{L^2}+\int_{\mathbb{R}^{d}}P'(1+\rho)\nabla\Lambda^1 \rho \Lambda^1 u dx \\ \notag
&+\mu\|\nabla\Lambda^1 u\|^2_{L^2}+(\mu+\mu')\|div\Lambda^1 u\|^2_{L^2}-\int_{\mathbb{R}^{d}}div\Lambda^1 \tau \Lambda^1 u dx  \\ \notag
&=\frac 1 2 \int_{\mathbb{R}^{d}} \partial_t\rho |\Lambda^1 u|^2 dx
-\int_{\mathbb{R}^{d}}\Lambda^1 u \cdot(1+\rho)u\cdot\nabla \Lambda^1 u dx  \\ \notag
&-\int_{\mathbb{R}^{d}}\Lambda^1 u \nabla u (1+\rho)\Lambda^1 u dx
-\int_{\mathbb{R}^{d}}\Lambda^1[h(\rho)-\gamma]\nabla\rho (1+\rho)\Lambda^1 u dx \\ \notag
&+\int_{\mathbb{R}^{d}}\Lambda^1[i(\rho)-1]div\Sigma{(u)} (1+\rho)\Lambda^1 u dx
+\int_{\mathbb{R}^{d}}\Lambda^1[i(\rho)-1]div~\tau (1+\rho)\Lambda^1 u dx.
\end{align}
Using Lemmas \ref{Lemma0}-\ref{Lemma2}, we deduce that
\begin{align*}
&\frac 1 2 \int_{\mathbb{R}^{d}} \partial_t\rho |\Lambda^1 u|^2 dx-\int_{\mathbb{R}^{d}}\Lambda^1[h(\rho)-\gamma]\nabla\rho (1+\rho)\Lambda^1 u dx
-\int_{\mathbb{R}^{d}}\Lambda^1 u\nabla u (1+\rho)\Lambda^1 u dx  \\ \notag
&\lesssim (\|u\|_{H^s}+\|\rho\|_{H^s})(\|\Lambda^2 \rho\|^2_{L^2}+\|\Lambda^2 u\|^2_{L^2}),
\end{align*}
and
\begin{align*}
&\int_{\mathbb{R}^{d}}\Lambda^1[i(\rho)-1]div~\tau (1+\rho)\Lambda^1 u dx
+\int_{\mathbb{R}^{d}}\Lambda^1[i(\rho)-1]div\Sigma{(u)} (1+\rho)\Lambda^1 u dx  \\ \notag
&\lesssim(\|u\|_{H^s}+\|\rho\|_{H^s})(\|\Lambda^2 \rho\|^2_{L^2}+\|\Lambda^2 u\|^2_{L^2}+\|\nabla_R \nabla g\|^2_{L^{2}(\mathcal{L}^{2})}).
\end{align*}
Integrating by part, we get
\begin{align*}
&-\int_{\mathbb{R}^{d}}\Lambda^1 u \cdot(1+\rho)u\cdot\nabla \Lambda^1 u dx=\frac 1 2 \int_{\mathbb{R}^{d}}div[(1+\rho)u]|\Lambda^1 u|^2 dx \\
&\lesssim(\|u\|_{H^s}+\|\rho\|_{H^s})(\|\Lambda^2 \rho\|^2_{L^2}+\|\Lambda^2 u\|^2_{L^2}),
\end{align*}
and
\begin{align*}
&-\int_{\mathbb{R}^{d}}P'(1+\rho)(\Lambda^1 u\nabla\Lambda^1\rho+\Lambda^1\rho div\Lambda^1 u)dx
=\int_{\mathbb{R}^{d}}P''(1+\rho)\Lambda^1\rho \Lambda^1 u\nabla\rho dx  \\
&\lesssim(\|u\|_{H^s}+\|\rho\|_{H^s})(\|\Lambda^2 \rho\|^2_{L^2}+\|\Lambda^2 u\|^2_{L^2}).
\end{align*}
Multiplying $\nabla\Lambda^{1} \rho$ to $(\ref{r2})$ and integrating over $\mathbb{R}^{d}$ with $x$, we get
\begin{align}\label{g5}
&\frac {d} {dt} \int_{\mathbb{R}^{d}}\Lambda^{1} u \cdot\nabla\Lambda^{1} \rho dx
+\gamma\|\nabla\Lambda^{1}\rho\|^2_{L^{2}}
=-\int_{\mathbb{R}^{d}} \Lambda^{1} \rho_t div\Lambda^{1} u dx \\ \notag
&-\int_{\mathbb{R}^{d}}\nabla\Lambda^{1} \rho\cdot u\cdot\nabla \Lambda^{1} u dx
-\int_{\mathbb{R}^{d}}\Lambda^{1} u\nabla u \nabla\Lambda^{1} \rho dx-\int_{\mathbb{R}^{d}}\Lambda^{1}((h(\rho)-\gamma)\nabla\rho) \nabla\Lambda^{1} \rho dx \\ \notag
&+\int_{\mathbb{R}^{d}}\Lambda^{1}(i(\rho)div\Sigma{(u)}) \nabla\Lambda^{1} \rho dx
+\int_{\mathbb{R}^{d}}\Lambda^{1}(i(\rho)div\tau) \nabla\Lambda^{1} \rho dx.
\end{align}
Using Lemma \ref{Lemma0}, we deduce that
\begin{align*}
&-\int_{\mathbb{R}^{d}} \Lambda^{1} \rho_t div\Lambda^{1} u dx
-\int_{\mathbb{R}^{d}}\nabla\Lambda^{1} \rho\cdot u\cdot\nabla \Lambda^{1} u dx-\int_{\mathbb{R}^{d}}\Lambda^{1} u\nabla u \nabla\Lambda^{1} \rho dx\\
&\lesssim (\|u\|_{H^s}+\|\rho\|_{H^s})(\|\Lambda^2 \rho\|^2_{L^2}+\|\Lambda^2 u\|^2_{L^2})+\|\nabla\Lambda^{1} u\|^2_{L^2},
\end{align*}
and
\begin{align*}
&-\int_{\mathbb{R}^{d}}\Lambda^{1}((h(\rho)-\gamma)\nabla\rho) \nabla\Lambda^{1} \rho dx +\int_{\mathbb{R}^{d}}\Lambda^{1}(i(\rho)div\Sigma{(u)}) \nabla\Lambda^{1} \rho dx\\
&\lesssim \|\rho\|_{H^{s}}\|\nabla\Lambda^{1} \rho\|^2_{L^2}+\|\nabla\Lambda^{1} \rho\|_{L^2}(\|\nabla^2 u\|_{H^{1}}+\|\nabla^2 u\|_{L^2}\|\rho\|_{H^{s}}).
\end{align*}
Using Lemma \ref{Lemma1} and Lemma \ref{Lemma2}, we have
\begin{align*}
\int_{\mathbb{R}^{d}}\Lambda^{1}(i(\rho)div\tau) \nabla\Lambda^{1} \rho dx\lesssim \|\nabla\Lambda^{1} \rho\|_{L^2}\|\nabla_R \nabla g\|_{H^{1}(\mathcal{L}^{2})}(\|\rho\|_{H^{s}}+1).
\end{align*}
Combining \eqref{g2} and the estimates for \eqref{g3}-\eqref{g5}, we deduce that
\begin{align}\label{mid estimate}
&\frac {d} {dt} (\|h(\rho)^{\frac 1 2}\Lambda^1\rho \|^2_{L^2}+\|(1+\rho)^{\frac 1 2}\Lambda^1 u\|^2_{L^2}+\|\Lambda^1 g\|^2_{L^2(\mathcal{L}^{2})}+2\eta\int_{\mathbb{R}^{d}} \Lambda^{1} u\nabla\Lambda^{1} \rho dx)  \\ \notag
&+2(\mu\|\nabla \Lambda^1 u\|^2_{L^2}+(\mu+\mu')\|div\Lambda^1 u\|^2_{L^2}+
\eta\gamma\|\nabla\Lambda^{1}\rho\|^2_{L^2}+\|\nabla_R \Lambda^1 g\|^2_{H^1(\mathcal{L}^{2})})  \\ \notag
&\lesssim (\|\rho\|_{H^{s}}+\|u\|_{H^{s}}+\|g\|_{H^s(\mathcal{L}^{2})})(\|\nabla^2 \rho\|^2_{L^{2}}+\|\nabla^2 u\|^2_{H^{1}}+\|\nabla_R\Lambda^1 g\|^2_{L^{2}(\mathcal{L}^{2})})  \\ \notag
&+\eta(\|\nabla^2 u\|^2_{L^{2}}+\|\nabla^2 u\|_{H^{1}}\|\nabla^2 \rho\|_{L^{2}}+\|\nabla^2 \rho\|_{L^{2}}\|\nabla\nabla_R g\|_{H^{1}(\mathcal{L}^{2})}).
\end{align}
Choosing $\epsilon$ and $\eta$ small enough, the estimates \eqref{mid estimate} and \eqref{high estimate} ensure that
\begin{align}\label{ineq14}
\frac d {dt} E^1_\eta(t)+D^1_\eta (t)\leq 0.
\end{align}
From the above inequality, we deduce that
\begin{align}\label{ineq15}
\begin{split}
&\frac d {dt} E^1_\eta+\frac { C_d} {1+t}(\mu\|\Lambda^1 u\|^2_{H^{s-1}}+\eta\gamma\|\Lambda^1 \rho\|^2_{H^{s-2}})
+\|\Lambda^1 \nabla_R g\|^2_{H^{s-1}(\mathcal{L}^{2})} \\
&\leq \frac {CC_d} {1+t}\int_{S(t)}|\xi|^2(|\hat{u}(\xi)|^2+|\hat{\rho}(\xi)|^2) d\xi.
\end{split}
\end{align}
According to \eqref{ineq13}, we have
\begin{align*}
\frac {CC_d} {1+t}\int_{S(t)}|\xi|^2(|\hat{u}(\xi)|^2+|\hat{\rho}(\xi)|^2) d\xi\leq C{C_d}^2 (1+t)^{-2}(\|\rho\|^2_{L^2}+\|u\|^2_{L^2})\leq C (1+t)^{-\frac d 4-2}.
\end{align*}
Then the proof of \eqref{ineq11} implies that $E^1_\eta \leq C(1+t)^{-\frac d 4-1}$.
We thus complete the proof of Proposition \ref{pro4}.
\end{proof}

By virtue of the decay rate for $E(t)$ and $E^1_\eta(t)$, we can show that the solution of \eqref{eq1} belongs to some Besov space with negative index.
\begin{prop}\label{pro5}
Let $(\rho_0,u_0,g_0)$ satisfy the same condition in Theorem \ref{th2}. Then the corresponding solution
\begin{align}\label{ineq16}
(\rho,u,g)\in L^{\infty}(0,\infty;\dot{B}^{-\frac d 2}_{2,\infty})\times L^{\infty}(0,\infty;\dot{B}^{-\frac d 2}_{2,\infty})\times L^{\infty}(0,\infty;\dot{B}^{-\frac d 2}_{2,\infty}(\mathcal{L}^2)).
\end{align}
\begin{proof}
Applying $\dot{\Delta}_j$ to the system \eqref{eq1}, we get
\begin{align}\label{eq10}
\left\{
\begin{array}{ll}
\dot{\Delta}_j\rho_t+div~\dot{\Delta}_j u=\dot{\Delta}_j F,  \\[1ex]
\dot{\Delta}_j u_t-div\Sigma(\dot{\Delta}_j u)+\gamma\nabla\dot{\Delta}_j \rho-div\dot{\Delta}_j\tau=\dot{\Delta}_j G,  \\[1ex]
\dot{\Delta}_j g_t+\mathcal{L}\dot{\Delta}_j g-\nabla\dot{\Delta}_j u R_j \partial_{R_k}\mathcal{U}+div \dot{\Delta}_j u=\dot{\Delta}_j H, \\[1ex]
\end{array}
\right.
\end{align}
where $F=-div(\rho u)$, $G=-u\cdot\nabla u+[i(\rho)-1](div\Sigma(u)+div \tau)+[\gamma-h(\rho)]\nabla\rho$ and $H=-u\cdot\nabla g-\frac 1 {\psi_\infty} \nabla_R\cdot(\nabla u Rg\psi_\infty)$.

Using the fact that $\int_{B} \dot{\Delta}_j g\psi_\infty dR=0$ and integrating by parts, we obtain
\begin{align}\label{ineq17}
&\frac 1 2 \frac d {dt}(\gamma\|\dot{\Delta}_j \rho\|^2_{L^2}+\|\dot{\Delta}_j u\|^2_{L^2}+\|\dot{\Delta}_j g\|^2_{L^2(\mathcal{L}^{2})}) \\ \notag
&+\mu\|\nabla\dot{\Delta}_j u\|^2_{L^2}+(\mu+\mu')\|div \dot{\Delta}_j u\|^2_{L^2}+\|\nabla_R \dot{\Delta}_j g\|^2_{L^2(\mathcal{L}^{2})}   \\ \notag
&=\int_{\mathbb{R}^{d}} \gamma\dot{\Delta}_j F\dot{\Delta}_j \rho dx+\int_{\mathbb{R}^{d}} \dot{\Delta}_j G\dot{\Delta}_j u dx+\int_{\mathbb{R}^{d}}\int_{B} \dot{\Delta}_j H\dot{\Delta}_j g \psi_\infty dxdR   \\ \notag
&\leq C(\|\dot{\Delta}_j F\|_{L^2}\|\dot{\Delta}_j \rho\|_{L^2}+\|\dot{\Delta}_j G\|_{L^2}\|\dot{\Delta}_j u\|_{L^2}) \\ \notag
&+C(\|\dot{\Delta}_j (u\nabla g)\|^2_{L^2(\mathcal{L}^{2})})+\|\dot{\Delta}_j (\nabla uRg)\|^2_{L^2(\mathcal{L}^{2})})+\frac 1 2 \|\nabla_R \dot{\Delta}_j g\|^2_{L^2(\mathcal{L}^{2})} .
\end{align}
Multiplying both sides of \eqref{ineq17} by $2^{-jd}$ and taking $l^\infty$-norm, we get
\begin{align}\label{ineq18}
&\frac d {dt}(\gamma\|\rho\|^2_{\dot{B}^{-\frac d 2}_{2,\infty}}+\|u\|^2_{\dot{B}^{-\frac d 2}_{2,\infty}}+\|g\|^2_{\dot{B}^{-\frac d 2}_{2,\infty}(\mathcal{L}^{2})})  \\ \notag
&\leq C(\|F\|_{\dot{B}^{-\frac d 2}_{2,\infty}}\|\rho\|_{\dot{B}^{-\frac d 2}_{2,\infty}}+\|G\|_{\dot{B}^{-\frac d 2}_{2,\infty}}\|u\|_{\dot{B}^{-\frac d 2}_{2,\infty}}+\|u\nabla g\|^2_{\dot{B}^{-\frac d 2}_{2,\infty}(\mathcal{L}^{2})}+\|\nabla uRg\|^2_{\dot{B}^{-\frac d 2}_{2,\infty}(\mathcal{L}^{2})}).
\end{align}
Define $M(t)=\sup_{s\in[0,t]} \|\rho(s)\|_{\dot{B}^{-\frac d 2}_{2,\infty}}+\|u(s)\|_{\dot{B}^{-\frac d 2}_{2,\infty}}+\|g(s)\|_{\dot{B}^{-\frac d 2}_{2,\infty}(\mathcal{L}^{2})}$. According to \eqref{ineq18}, we deduce that
\begin{align}\label{ineq19}
M^2(t)&\leq CM^2(0)+M(t)\int_0^{t}\|F\|_{\dot{B}^{-\frac d 2}_{2,\infty}}+\|G\|_{\dot{B}^{-\frac d 2}_{2,\infty}}ds  \\ \notag
&+\int_0^{t}\|u\nabla g\|^2_{\dot{B}^{-\frac d 2}_{2,\infty}(\mathcal{L}^{2})}+\|\nabla uRg\|^2_{\dot{B}^{-\frac d 2}_{2,\infty}(\mathcal{L}^{2})}ds.
\end{align}
Using the fact that $L^1\hookrightarrow \dot{B}^{-\frac d 2}_{2,\infty}$ and the decay rates for $E$ and $E^1_\eta$, we obtain
\begin{align*}
\int_0^{t}\|u\nabla g\|^2_{\dot{B}^{-\frac d 2}_{2,\infty}(\mathcal{L}^{2})}+\|\nabla uRg\|^2_{\dot{B}^{-\frac d 2}_{2,\infty}(\mathcal{L}^{2})}ds
&\leq C\int_0^{t}\|u\nabla g\|^2_{L^1(\mathcal{L}^{2})}+\|\nabla uRg\|^2_{L^1(\mathcal{L}^{2})}ds  \\
&\leq C\int_0^{t}\|\nabla u\|^2_{L^2}\|g\|^2_{L^2(\mathcal{L}^{2})}+\|u\|^2_{L^2}\|\nabla g\|^2_{L^2(\mathcal{L}^{2})}ds\leq C,
\end{align*}
and
\begin{align*}
\int_0^{t}\|F\|_{\dot{B}^{-\frac d 2}_{2,\infty}}ds
&\leq C\int_0^{t}\|F\|_{L^1}ds \leq C\int_0^{t}\|u\|_{L^2}\|\nabla\rho\|_{L^2}+\|div u\|_{L^2}\|\rho\|_{L^2}ds   \\
&\leq C\int_0^{t}(1+s)^{-\frac d 8-\frac d 8-\frac 1 2}ds \leq C.
\end{align*}
By virtue of remark \ref{rema}, we get
\begin{align}\label{ineq20}
\int_0^{t}\|G\|_{L^1}ds
&\leq C\int_0^{t}(1+s)^{-\frac d 8-\frac d 8-\frac 1 2}ds+C\int_0^{t}\|div~\tau\|_{L^2}\|\rho\|_{L^2}ds  \\ \notag
&\leq C+C\int_0^{t}\|\rho\|_{L^2}\|\nabla g\|^{\frac 1 2}_{L^2(\mathcal{L}^{2})}\|\nabla\nabla_R g\|^{\frac 1 2}_{L^2(\mathcal{L}^{2})}ds  \\ \notag
&\leq C+C(\int_0^{t}\|\rho\|^{\frac 4 3}_{L^2}\|\nabla g\|^{\frac 2 3}_{L^2(\mathcal{L}^{2})}ds)^{\frac 3 4}(\int_0^{t}\|\nabla\nabla_R g\|^2_{L^2(\mathcal{L}^{2})}ds)^{\frac 1 2}     \\ \notag
&\leq C+C(\int_0^{t}(1+s)^{-\frac d 8\times\frac 4 3-(\frac d 8+\frac 1 2)\times\frac 2 3}ds)^{\frac 3 4}\leq C.
\end{align}
From \eqref{ineq19}, we have $M^2(t)\leq CM^2(0)+M(t)C+C,$ which implies that $M(t)\leq C$. We thus complete the proof.
\end{proof}
\end{prop}

From now on, we are going to finish the proof of theorem \ref{th2}.
\begin{proof}
Using \eqref{ineq13}, we obtain
\begin{align*}
\int_{S(t)}\int_{0}^{t}|\widehat{\rho u}|^2dsd\xi
&\leq C(1+t)^{-\frac d 2} \int_{0}^{t}\|\rho\|^2_{L^{2}}\|u\|^2_{L^{2}}ds \\ \notag
&\leq C(1+t)^{-\frac d 2}.
\end{align*}
Using \eqref{ineq20}, by Propositions \ref{pro0} and \ref{pro5}, we have
\begin{align*}
\int_{S(t)}\int_{0}^{t}|\hat{G}\cdot\bar{\hat{u}}|dsd\xi
&\leq C\int_{0}^{t}\|G\|_{L^{1}}\int_{S(t)}|\hat{u}|d\xi ds \\
&\leq C(1+t)^{-\frac d 4}\int_{0}^{t}\|G\|_{L^{1}}(\int_{S(t)}|\hat{u}|^2d\xi)^{\frac 1 2} ds  \\
&\leq C(1+t)^{-\frac d 2}M(t)\int_{0}^{t}\|G\|_{L^{1}}ds  \\
&\leq C(1+t)^{-\frac d 2}.
\end{align*}
Similar to the proof of Proposition \ref{pro3}, we have
\begin{align}\label{ineq21}
\frac d {dt} E_\eta(t)+\frac {\mu C_d} {1+t}\| u\|^2_{H^s}+\frac {\eta\gamma C_d} {1+t}\|\rho\|^2_{H^{s-1}}
+\|\nabla_R g\|^2_{H^{s}(\mathcal{L}^{2})}\leq \frac {CC_d} {1+t}(1+t)^{-\frac d 2}.
\end{align}
Then the proof of \eqref{ineq11} implies that
\begin{align}\label{ineq22}
E(t)\leq CE_\eta(t)\leq C(1+t)^{-\frac d 2}.
\end{align}
We now consider the faster decay rate for $E^1_\eta (t)$.
Recall that
\begin{align*}
&\frac d {dt} E^1_\eta+\frac { C_d} {1+t}(\mu\|\Lambda^1 u\|^2_{H^{s-1}}+\eta\gamma\|\Lambda^1 \rho\|^2_{H^{s-2}})
+\|\Lambda^1 \nabla_R g\|^2_{H^{s-1}(\mathcal{L}^{2})}  \\
&\leq \frac {CC_d} {1+t}\int_{S(t)}|\xi|^2(|\hat{u}(\xi)|^2+|\hat{\rho}(\xi)|^2) d\xi.
\end{align*}
According to \eqref{ineq22}, we have
\begin{align*}
\frac {CC_d} {1+t}\int_{S(t)}|\xi|^2(|\hat{u}(\xi)|^2+|\hat{\rho}(\xi)|^2) d\xi\leq C{C_d}^2 (1+t)^{-2}(\|\rho\|^2_{L^2}+\|u\|^2_{L^2})\leq C (1+t)^{-\frac d 2-2}.
\end{align*}
Then the proof of \eqref{ineq11} implies that $E^1_\eta \leq C(1+t)^{-\frac d 2-1}$.
Using \eqref{co1}, we get
\begin{align*}
\frac {1} {2}\frac {d} {dt} \|g\|^2_{L^2(\mathcal{L}^{2})}+\|\nabla_R g\|^2_{L^2(\mathcal{L}^{2})}\lesssim -\int_{\mathbb{R}^{d}}\nabla u:\tau dx+\|\nabla u\|_{L^\infty}\|g\|_{L^2(\mathcal{L}^{2})}\|\nabla_R g\|_{L^2(\mathcal{L}^{2})}.
\end{align*}
By Lemmas \ref{Lemma1},\ref{Lemma2} and Theorem \ref{th1}, we deduce that
\begin{align*}
\frac d {dt} \|g\|^2_{L^2(\mathcal{L}^2)}+\|g\|^2_{L^2(\mathcal{L}^2)}\leq C\|\nabla u\|^2_{L^2},
\end{align*}
from which we deduce that
\begin{align*}
\|g\|^2_{L^2(\mathcal{L}^2)}
&\leq \|g_0\|^2_{L^2(\mathcal{L}^2)}e^{-t}+C\int_{0}^{t}e^{-(t-s)}\|\nabla u\|^2_{L^2}ds  \\
&\leq C(e^{-t}+\int_{0}^{t}e^{-(t-s)}(1+s)^{-\frac d 2-1}ds)  \\
&\leq C(1+t)^{-\frac d 2-1},
\end{align*}
where in the last inequality we have used the fact that
\begin{align*}
&\lim_{t\rightarrow\infty}(1+t)^{\frac d 2+1}\int_{0}^{t}e^{-(t-s)}(1+s)^{-\frac d 2-1}ds = \lim_{t\rightarrow\infty}\frac {(1+t)^{\frac d 2+1}\int_{0}^{t}e^{s}(1+s)^{-\frac d 2-1}ds} {e^{t}}  \\
&=1+\lim_{t\rightarrow\infty}\frac {(\frac d 2+1)(1+t)^{\frac d 2}\int_{0}^{t}e^{s}(1+s)^{-\frac d 2-1}ds} {e^{t}}   \\
&=1.
\end{align*}
We thus complete the proof of Theorem \ref{th2}.
\end{proof}

\begin{rema}
One can see that the decay rate for the $\dot{H}^1$-norm obtained in Proposition \ref{pro4} is not optimal. However, in the proof of Theorem \ref{th2}, we improve the decay rate to $(1+t)^{-\frac{d}{4}-\frac{1}{2}}$.
\end{rema}

\begin{rema}
In Theorem \ref{th2}, we only obtain the optimal decay rate with $d\geq 3$. The decay rate for $d\leq 2$ is an interesting problem. However, the technique in this paper fail to obtain the optimal decay rate when $d\leq 2$. We are going to study about this problem in the future.
\end{rema}

\smallskip
\noindent\textbf{Acknowledgments} This work was
partially supported by the National Natural Science Foundation of China (No.11671407 and No.11701586), the Macao Science and Technology Development Fund (No. 098/2013/A3), and Guangdong Province of China Special Support Program (No. 8-2015),
and the key project of the Natural Science Foundation of Guangdong province (No. 2016A030311004).


\phantomsection
\addcontentsline{toc}{section}{\refname}
\bibliographystyle{abbrv} 
\bibliography{Feneref}

\begin{thebibliography}{10}

\bibitem{Bahouri2011}
H.~Bahouri, J.-Y. Chemin, and R.~Danchin.
\newblock {\em Fourier analysis and nonlinear partial differential equations},
  volume 343 of {\em Grundlehren der Mathematischen Wissenschaften}.
\newblock Springer, Heidelberg, 2011.

\bibitem{Bird1977}
R.~B. Bird, R.~C. Armstrong, and O.~Hassager.
\newblock {\em Dynamics of Polymeric Liquids}, volume~1.
\newblock Wiley, New York, 1977.

\bibitem{miaocompressns2}
D.~Bresch and B.~Desjardins.
\newblock Existence of global weak solutions for a 2d viscous shallow water
  equations and convergence to the quasi-geostrophic model.
\newblock {\em Communications in Mathematical Physics.}, 238(1-2):211--233,
  2003.

\bibitem{miaocompressns3}
D.~Bresch, B.~Desjardins, and C.~K. Lin.
\newblock On some compressible fluid models: Korteweg, lubrication, and shallow
  water systems.
\newblock {\em Communications in Partial Differential Equationss.},
  28(3-4):843--868, 2003.

\bibitem{Busuioc}
A.~V. Busuioc, I.~S. Ciuperca, D.~Iftimie, and L.~I. Palade.
\newblock The {FENE} dumbbell polymer model: existence and uniqueness of
  solutions for the momentum balance equation.
\newblock {\em J. Dynam. Differential Equations}, 26(2):217--241, 2014.

\bibitem{miao}
Q.~Chen, C.~Miao, and Z.~Zhang.
\newblock Well-posedness in critical spaces for the compressible
  {N}avier-{S}tokes equations with density dependent viscosities.
\newblock {\em Revista Matematica Iberoamericana.}, 26(3):915--946, 2010.

\bibitem{miaoill-posedness}
Q.~Chen, C.~Miao, and Z.~Zhang.
\newblock On the ill-posedness of the compressible {N}avier-{S}tokes equations
  in the critical besov spaces.
\newblock {\em Revista Matematica Iberoamericana.}, 31(4):165--186, 2011.

\bibitem{miaocompressns11}
R.~Danchin.
\newblock Global existence in critical spaces for compressible
  {N}avier-{S}tokes equations.
\newblock {\em Inventiones Mathematicae.}, 141(3):579--614, 2000.

\bibitem{miaocompressns12}
R.~Danchin.
\newblock Local theory in critical spaces for compressible viscous and
  heat-conductive gases.
\newblock {\em Comm Partial Differential Equations.}, 141(26):1183--1233, 2001.

\bibitem{miaocompressns14}
R.~Danchin.
\newblock On the uniqueness in critical spaces for compressible
  {N}avier-{S}tokes equations.
\newblock {\em Nonlinear Differential Equations Applications Nodea.},
  12(1):111--128, 2005.

\bibitem{miaocompressns15}
R.~Danchin.
\newblock Well-posedness in critical spaces for barotropic viscous fluids with
  truly not constant density.
\newblock {\em Communications in Partial Differential Equations.},
  32(9):111--128, 2007.

\bibitem{miaocompressns155}
R.~Danchin and P.~B. Mucha.
\newblock A lagrangian approach for the incompressible {N}avier-{S}tokes
  equations with variable density.
\newblock {\em Communications on Pure and Applied Mathematics.}, 65(10), 2012.

\bibitem{Doi1988}
M.~Doi and S.~F. Edwards.
\newblock {\em The Theory of Polymer Dynamics}.
\newblock Oxford University Press, Oxford, 1988.

\bibitem{miaocompressns24}
L.~He, J.~Huang, and C.~Wang.
\newblock Global stability of large solutions to the 3d compressible
  {N}avier-{S}tokes equations.
\newblock {\em Archive for Rational Mechanics and Analysis.},
  234(3):1167--1222, 2019.

\bibitem{He2009}
L.~He and P.~Zhang.
\newblock {$L^2$} {D}ecay of {S}olutions to a {M}icro-{M}acro {M}odel for
  {P}olymeric {F}luids {N}ear {E}quilibrium.
\newblock {\em Siam Journal on Mathematical Analysis}, 40(5):1905--1922, 2009.

\bibitem{Huang}
X.~Huang and J.~Li.
\newblock Existence and blowup behavior of global strong solutions to the
  two-dimensional barotrpic compressible {N}avier-{S}tokes system with vacuum
  and large initial data.
\newblock {\em J. Math. Pures Appl. (9)}, 106(1):123--154, 2016.

\bibitem{miaocompressns25}
X.~Huang, J.~Li, and Z.~Xin.
\newblock Global well-posedness of classical solutions with large oscillations
  and vacuum to the three-dimensional isentropic compressible {N}avier-{S}tokes
  equations.
\newblock {\em Communications on Pure Applied Mathematics.}, 234, 2011.

\bibitem{2017Global}
N.~Jiang, Y.~Liu, and T.-F. Zhang.
\newblock Global classical solutions to a compressible model for micro-macro
  polymeric fluids near equilibrium.
\newblock {\em SIAM J. Math. Anal.}, 50(4):4149--4179, 2018.

\bibitem{miaocompressns26}
Q.~Jiu, Y.~Wang, and Z.~Xin.
\newblock Global well-posedness of 2d compressible {N}avier-{S}tokes equations
  with large data and vacuum.
\newblock {\em Journal of Mathematical Fluid Mechanics.}, 16(3):483--521, 2014.

\bibitem{Jourdain}
B.~Jourdain, T.~Leli\`evre, and C.~Le~Bris.
\newblock Existence of solution for a micro-macro model of polymeric fluid: the
  {FENE} model.
\newblock {\em J. Funct. Anal.}, 209(1):162--193, 2004.

\bibitem{miaocompressns18}
A.~V. Kazhikhov and V.~V. Shelukhin.
\newblock Unique global solution with respect to time of initial-boundary value
  problems for one-dimensional equations of a viscous gas.
\newblock {\em J. Appl. Math. Mech.}, 41(2):273--282., 1977.

\bibitem{miaocompressns28}
H.~Li, J.~Li, and Z.~Xin.
\newblock Vanishing of vacuum states and blow-up phenomena of the compressible
  {N}avier-{S}tokes equations.
\newblock {\em Communications in Mathematical Physics.}, 281(2):401, 2008.

\bibitem{Li2011Large}
H.-L. Li and T.~Zhang.
\newblock Large time behavior of isentropic compressible {N}avier-{S}tokes
  system in {$\Bbb R^3$}.
\newblock {\em Math. Methods Appl. Sci.}, 34(6):670--682, 2011.

\bibitem{F.Lin}
F.~Lin, P.~Zhang, and Z.~Zhang.
\newblock On the global existence of smooth solution to the 2-{D} {FENE}
  dumbbell model.
\newblock {\em Comm. Math. Phys.}, 277(2):531--553, 2008.

\bibitem{Luo-Yin}
W.~Luo and Z.~Yin.
\newblock The {L}iouville {T}heorem and the {$L^2$} {D}ecay for the {FENE}
  {D}umbbell {M}odel of {P}olymeric {F}lows.
\newblock {\em Arch. Ration. Mech. Anal.}, 224(1):209--231, 2017.

\bibitem{Luo-Yin2}
W.~Luo and Z.~Yin.
\newblock The {$L^2$} decay for the 2{D} co-rotation {FENE} dumbbell model of
  polymeric flows.
\newblock {\em Adv. Math.}, 343:522--537, 2019.

\bibitem{2021Global}
Z.~Luo, W.~Luo, and Z.~Yin.
\newblock Global strong solutions and large time behavior to the compressible
  co-rotation {FENE} dumbbell model of polymeric flows near equilibrium.
\newblock arXiv:2104.10844.

\bibitem{Masmoudi2008}
N.~Masmoudi.
\newblock Well-posedness for the {FENE} dumbbell model of polymeric flows.
\newblock {\em Comm. Pure Appl. Math.}, 61(12):1685--1714, 2008.

\bibitem{Masmoudi2013}
N.~Masmoudi.
\newblock Global existence of weak solutions to the {FENE} dumbbell model of
  polymeric flows.
\newblock {\em Invent. Math.}, 191(2):427--500, 2013.

\bibitem{2016Equations}
N.~Masmoudi.
\newblock {\em Equations for Polymeric Materials}.
\newblock Handbook of Mathematical Analysis in Mechanics of Viscous Fluids,
  2016.

\bibitem{Matsumura}
A.~Matsumura and T.~Nishida.
\newblock The initial value problem for the equations of motion of compressible
  viscous and heat-conductive fluids.
\newblock {\em Proc. Japan Acad. Ser. A Math. Sci.}, 55:337--342, 1979.

\bibitem{miaocompressns21}
A.~Mellet and A.~Vasseur.
\newblock On the barotropic compressible {N}avier-{S}tokes equations.
\newblock {\em Proc. Japan Acad. Ser. A Math. Sci.}, 32(3):431--452, 2007.

\bibitem{Moser1966A}
J.~Moser.
\newblock A rapidly convergent iteration method and non-linear partial
  differential equations. {I}.
\newblock {\em Ann. Scuola Norm. Sup. Pisa Cl. Sci. (3)}, 20:265--315, 1966.

\bibitem{miaocompressns23}
J.~Nash.
\newblock Le probleme de cauchy pour les equations differentielles d'un fluide
  general.
\newblock {\em Bull.soc.math.france.}, 90(4):487--497, 1962.

\bibitem{1959On}
L.~Nirenberg.
\newblock On elliptic partial differential equations.
\newblock {\em Ann.scuola Norm.sup.pisa}, 1959.

\bibitem{Renardy}
M.~Renardy.
\newblock An existence theorem for model equations resulting from kinetic
  theories of polymer solutions.
\newblock {\em SIAM J. Math. Anal.}, 22(2):313--327, 1991.

\bibitem{Schonbek1985}
M.~E. Schonbek.
\newblock {$L^2$} decay for weak solutions of the {N}avier-{S}tokes equations.
\newblock {\em Arch. Rational Mech. Anal.}, 88(3):209--222, 1985.

\bibitem{Schonbek}
M.~E. Schonbek.
\newblock Existence and decay of polymeric flows.
\newblock {\em SIAM J. Math. Anal.}, 41(2):564--587, 2009.

\bibitem{Tong2017The}
L.~Tong, Z.~Tan, and Y.~Wang.
\newblock The asymptotic behavior of globally smooth solutions to the
  compressible magnetohydrodynamic equations with {C}oulomb force.
\newblock {\em Analysis and Applications}, 2017.

\bibitem{Vasseur}
A.~F. Vasseur and C.~Yu.
\newblock Existence of global weak solutions for 3{D} degenerate compressible
  {N}avier-{S}tokes equations.
\newblock {\em Invent. Math.}, 206(3):935--974, 2016.

\bibitem{Xin98}
Z.~Xin.
\newblock Blowup of smooth solutions to the compressible {N}avier-{S}tokes
  equation with compact density.
\newblock {\em Communications on Pure Applied Mathematics.}, 1998.

\bibitem{miaocompressns27}
Z.~Xin and W.~Yan.
\newblock On blow-up of classical solutions to the compressible
  {N}avier-{S}tokes equations.
\newblock {\em Communications in Mathematical Physics.}, 321(2):529--541, 2013.

\bibitem{Xu2019}
J.~Xu.
\newblock A low-frequency assumption for optimal time-decay estimates to the
  compressible {N}avier-{S}tokes equations.
\newblock {\em Comm. Math. Phys.}, 371(2):525--560, 2019.

\bibitem{Zhang-H}
H.~Zhang and P.~Zhang.
\newblock Local existence for the {FENE}-dumbbell model of polymeric fluids.
\newblock {\em Arch. Ration. Mech. Anal.}, 181(2):373--400, 2006.

\bibitem{2018Global}
Z.~Zhou, C.~Zhu, and R.~Zi.
\newblock Global well-posedness and decay rates for the three dimensional
  compressible {O}ldroyd-{B} model.
\newblock {\em Journal of Differential Equations}, 265(4), 2018.

\end{thebibliography}

\end{document}